\documentclass[12pt,reqno]{amsart}
\usepackage{amsfonts, color}
\usepackage{amsmath}
\usepackage{amssymb}
\usepackage{mathrsfs}
\usepackage[colorlinks=true]{hyperref}
\usepackage{enumitem}
\usepackage{csquotes}

\usepackage{geometry}
\newgeometry{
    inner=2cm,
    outer=2cm,
    top=2cm,
    bottom=2cm
}
\renewcommand\eqref[1]{(\ref{#1})}

\allowdisplaybreaks

\numberwithin{equation}{section}
\newtheorem{theorem}{Theorem}[section]
\newtheorem{definition}[theorem]{Definition}
\newtheorem{remark}[theorem]{Remark}
\newtheorem{lemma}[theorem]{Lemma}
\newtheorem{proposition}[theorem]{Proposition}

\newtheorem{corollary}[theorem]{Corollary}
\newtheorem*{theorem*}{Theorem}
\newtheorem*{assumption*}{Assumptions}

\newcommand{\Dim}{\mathrm{dim}}
\newcommand{\Rep}{\mathrm{Rep}}
\newcommand{\Hom}{\mathrm{Hom}}
\newcommand{\Aut}{\mathrm{Aut}}
\newcommand{\lv}{\left\|}
\newcommand{\rv}{\right\|}
\newcommand{\prv}{\right\|_{\ell^{p}(\widehat{G})}}
\newcommand{\psrv}{\right\|_{\ell_{sch}^{p}(\widehat{G})}}
\newcommand{\G}{\widehat{G}}

\title[]{On geometric properties of $\ell^{p}$-spaces on unitary duals of compact groups}
\author[M. Akhymbek]{Meiram Akhymbek}
\address{
	Meiram Akhymbek:
	\endgraf
    Institute of Mathematics and Mathematical Modeling
    \endgraf
    28 Shevchenko str.
    \endgraf
    050010 Almaty
    \endgraf
    Kazakhstan
    \endgraf
	{\it E-mail address} {\rm m.akhymbek@gmail.com}
		}
    
\author[M. Ruzhansky]{Michael Ruzhansky}
\address{
  Michael Ruzhansky:
  \endgraf
  Department of Mathematics: Analysis, Logic and Discrete Mathematics
  \endgraf
  Ghent University, Belgium
  \endgraf
 and
  \endgraf
  School of Mathematical Sciences
  \endgraf
  Queen Mary University of London
  \endgraf
  United Kingdom
  \endgraf
  {\it E-mail address} {\rm michael.ruzhansky@ugent.be}
  }

\begin{document}


\begin{abstract} 
In this paper, we study geometric properties of $\ell^{p}$-spaces associated with the unitary dual of a compact group. More precisely, we prove uniform smoothness, uniform convexity, Clarkson type inequalities, Kadec-Klee property, as well as type and cotype properties of such spaces. We also present duality and complex interpolation results. 
\end{abstract}

\maketitle

\tableofcontents

\section{Introduction}

The notion of uniformly convex spaces was first introduced by Clarkson \cite{clarkson1936uniformly}. A Banach space $X$ is said to be uniformly convex if for each $0<\varepsilon\leq 2$ there exists $\delta(\varepsilon)>0$ such that 
$$\lv x\rv=\lv y\rv=1,\quad \lv x-y\rv=\varepsilon$$ 
implies
$$\lv \frac{x+y}{2}\rv\leq 1-\delta(\varepsilon).$$
It can shortly be understood in geometrical terms as follows: the mid-point of an arbitrary chord of the unit sphere of the space cannot approach to the surface of that sphere unless the length of the chord goes to zero \cite{clarkson1936uniformly}.

Note that any finite dimensional Euclidean space $\mathbb{R}^{n},$ $n\geq 1,$ and any Hilbert space $\mathcal{H}$ is clearly uniformly convex due to the parallelogram identity 
\begin{equation}\label{parallelogram identity}
\lv x+y\rv^2+\lv x-y\rv^2=2(\lv x\rv^2+\lv y\rv^2).
\end{equation}
In the same paper \cite[Section 3]{clarkson1936uniformly}, Clarkson proved that the classical Lebesgue spaces $L^{p}(\mu)$ and $\ell^{p},$ for $1<p<\infty$, satisfy this property too, i.e., they are uniformly convex. 

Another important notion is the uniform smoothness of a Banach space, which is a dual property of the uniform convexity. A Banach space $X$ is said to be uniformly smooth if the expression 
$$\sup\{1-\lv \frac{x+y}{2}\rv:\quad \lv x\rv=\lv y\rv=1,\quad\lv x-y\rv\leq 2\tau\}$$
equals to $o(\tau)$ as $\tau\to 0.$

Equivalent definitions of uniform convexity and uniform smoothness via the modulus of convexity and the modulus of smoothness, respectively, can be found in Section \ref{section type} (Definitions \ref{modulus} and \ref{uniform convexity and smoothness}).

In this paper, we investigate the uniform convexity and uniform smoothness of noncommutative $\ell^p$-spaces on unitary dual of a compact group $G$ based on the Schatten-von Neumann ideals (see, Theorem \ref{mainth0}), simply denoted as $\ell_{sch}^{p}(\widehat{G})$ \cite[Section 2.1.4]{fischer2016quantization} (see also \cite[Section 2.14.2]{edwards1972integration}, \cite{hewitt2013abstract}). These $\ell_{sch}^{p}(\widehat{G})$ spaces based on Schatten-von Neumann ideals are the generalization of $\ell^p$-spaces over the compact group $G,$ denoted as $\ell^{p}(G)$ (see \cite{edwards1972integration}, \cite{fischer2016quantization}, \cite{hewitt2013abstract}, \cite{ruzhansky2010pseudo}). One of the known applications of $\ell_{sch}^{p}(\G)$ spaces is the Hausdorff-Young theorem for all compact groups \cite[Section 2.14.1]{edwards1972integration}.

In order to show the uniform convexity and uniform smoothness of $L^p(\mu)$ (or $\ell^p$) spaces, Clarkson proved various inequalities in these spaces close to the parallelogram identity \eqref{parallelogram identity} (see \cite[Theorem 2]{clarkson1936uniformly} for more details). These inequalities are known as Clarkson inequalities. Hence, our first aim is to obtain Clarkson type inequalities in spaces $\ell_{sch}^{p}(\G)$ (see Proposition \ref{new Clarkson type inequality}). Moreover, we prove complex interpolation (see, Subsection \ref{subsectioninterpolation}) and duality results (see, Subsection \ref{subsectiondual}) for the given spaces $\ell_{sch}^{p}(\G)$ similar to the classical Lebesgue spaces. These results act as important tools to obtain Clarkson type inequalities and have their own importance.
 
Another important classical geometrical property of Banach spaces is defining the type and cotype of the given Banach space. The type and cotype of a Banach space are closely related to the above mentioned uniform convexity and uniform smoothness (see, for example, \cite[Theorem 1.e.16]{lindenstrauss2013classical} and \cite{Figiel1976}). A Banach space is said to be of type $p$ for some $1\leq p\leq 2$ (resp. of cotype $q$ for some $2\leq q\leq \infty)$ if there exists a constant $M>0$ such that 
$$\left(\frac{1}{2^n}\sum_{\theta_{j}=\pm1}\lv \sum_{j=1}^{n}\theta_{j}x_{j}\rv^2\right)^{1/2}\leq M\left(\sum_{j=1}^{n}\lv x_{j}\rv^{p}\right)^{1/p},$$
$$\left(\text{resp.}\quad\left(\frac{1}{2^n}\sum_{\theta_{j}=\pm1}\lv \sum_{j=1}^{n}\theta_{j}x_{j}\rv^2\right)^{1/2}\geq M\left(\sum_{j=1}^{n}\lv x_{j}\rv^{q}\right)^{1/q}\right),$$
for every finite number of vectors $\{x_{j}\}_{j=1}^{n}$ in $X$, where the left hand side means the average value over all $2^{n}$ terms with possible combinations of $\theta_{j}=1$ and $\theta_{j}=-1$, $1\leq j\leq n.$

For example, it is well known that the classical Lebesgue space $L^p(\mu), 1\leq p<\infty,$ is of type $\min\{2,p\}$ and cotype $\max\{2,p\}$. Furthermore, any Hilbert space is obviously of type $2$ and cotype $2$ due to the parallelogram identity. The converse statement is also true, that any space which is simultaneously of type 2 and cotype 2, is isomorphic to a Hilbert space \cite[Proposition 3.1]{kwapien1972isomorphic}. In this paper, we prove that the spaces $\ell^{p}(\G)$ and $\ell_{sch}^{p}(\G)$ have the same type and cotype properties as the usual Lebesgue spaces $L^p(\mu)$.

It is a folklore that a Banach space is a Hilbert space if and only if one has the parallelogram identity. Hence, the notions of type and cotype measure how bad the parallelogram identity gets, i.e., how far a Banach space is from being a Hilbert space. Since in general the parallelogram identity does not hold in an arbitrary Banach space, one has to investigate its inequality form closely related to the above mentioned Clarkson inequalities. Hence, using the obtained Clarkson type inequalities, we also consider type and cotype properties of $\ell_{sch}^{p}(\G)$.

We also prove a modified version of the Clarkson type inequalities with some constants involved (see Theorem \ref{mainth2}).
The earlier versions of these inequalities in $L^{p}$ spaces and Schatten-von Neumann ideals $\mathcal{S}^p$ were used to confirm the conjecture by Gross, which arose in his work in quantum field theory (see \cite{carlen1993optimal} and \cite{ball2002sharp} for more details). Various other application of these type of inequalities can also be found in \cite{pisier1999volume}. Moreover, using the above mentioned results, we also present type and cotype properties of $\ell_{sch}^{p}(\G)$.

Furthermore, in the last section of the paper, we shortly present some of the above mentioned properties for a class of noncommutative $\ell^p$-spaces on unitary dual of compact group $G$, based on Hilbert-Schmidt ideal, denoted by $\ell^{p}(\G)$, introduced in \cite[Section 10.3.3]{ruzhansky2010pseudo} (see also \cite[Section 2.1.4]{fischer2016quantization} for more details on these spaces). The Clarkson type inequalities and the reflexivity result for the given space was earlier investigated by the authors in \cite{tulenov2018clarkson}.

\section{Preliminaries}

In this section, we recall the necessary preliminaries and the basics of the main object of this paper, the noncommutative $\ell^{p}$-spaces associated with a compact group. Throughout this paper a group is always assumed to be compact group (shortly, group), i.e., a topological group which is compact as a topological space (see, for example, \cite[Section III.7.3]{ruzhansky2010pseudo}). 

Let $G$ and $H$ be the groups. By $\mathrm{Hom}(G,H)$, we denote the set of all (group) homomorphisms from $G$ into $H.$ When $G=H$, we shortly write $\Hom(G)$ instead of $\Hom(G,G).$ The set of all bijective homomorphisms from $G$ onto $G$ is denoted by $\mathrm{Aut}(G).$

Throughout this paper, by $\left(\mathcal{L}(\mathcal{H}), \lv \cdot\rv_{\mathcal{L}(\mathcal{H})}\right)$ and $\left(\mathcal{S}^{p}(\mathcal{H}), \lv\cdot\rv_{\mathcal{S}^{p}(\mathcal{H})}\right),$ $1\leq p\leq \infty$, we denote the $\ast$-algebra of all bounded linear operators on a Hilbert space $\mathcal{H}$ and a Schatten-von Neumann ideal of compact operators on $\mathcal{H}$, respectively (see, \cite[Chapter 2]{simon2005trace} or \cite[Chapter 3]{gohberg1969translations}.  

\subsection{Strongly continuous irreducible unitary representations}

Let $\mathcal{H}$ be a complex Hilbert space equipped with an inner product $\langle\cdot,\cdot \rangle_{\mathcal{H}}$. The unitary group of $\mathcal{H}$ is 
$$\mathcal{U}(\mathcal{H}):=\{A\in\Aut(\mathcal{H})\mid \forall v,w\in \mathcal{H} : \langle Av, Aw\rangle_{\mathcal{H}}=\langle v,w\rangle_{\mathcal{H}}\},$$
i.e., $\mathcal{U}(\mathcal{H})$ contains the unitary linear bijections from $\mathcal{H}$ into $\mathcal{H}.$
A representation of a group $G$ on a Hilbert space $\mathcal{H}_{\phi}$ is any $\phi\in\mathrm{Hom}(G,\Aut(\mathcal{H}_{\phi}))$. A representation $\psi\in\Hom(G,\mathcal{U}(\mathcal{H}_{\phi}))$ is called a unitary representation.

Let $A$ be a bijective homomorphism from $\mathcal{H}$ onto $\mathcal{H}$, i.e., $A\in\Hom(\mathcal{H},\mathcal{H})$. A subspace $\mathcal{K}\subset \mathcal{H}$ is called $A$-invariant if $A\mathcal{K}\subset \mathcal{K}$.

Let $\phi\in\Hom(G,\Aut(H_{\phi})).$ A subspace $\mathcal{K}\subset \mathcal{H}_{\phi}$ is called $\phi$-invariant if $\mathcal{K}$ is $\phi(x)$-invariant for all $x\in G.$ A representation $\phi\in\Hom(G,\Aut(\mathcal{H}_{\phi}))$ is called irreducible if the trivial subspaces $\{0\}$ and $\mathcal{H}_{\phi}$ are the only $\phi$-invariant subspaces of $\mathcal{H}.$      

Let the unitary group $\mathcal{U}(H_{\phi})$ be endowed with a strong operator topology inherited from the class of all linear bounded operators $\mathcal{L}(H_{\phi})$. Then a representation $\phi\in\Hom(G,\mathcal{U}(\mathcal{H}_{\phi}))$ is strongly continuous if 
$$x\mapsto \phi(x) v: G\rightarrow \mathcal{H}_{\phi}$$
is continuous for all $v\in \mathcal{H}_{\phi}.$

By $\Rep(G)$, we denote the set of all strongly continuous
irreducible unitary representations of $G$. The strongly continuous irreducible unitary representations $\phi\in\Hom(G,\mathcal{U}(\mathcal{H}_{\phi}))$ and $\psi\in\Hom(G,\mathcal{U}(\mathcal{H}_{\psi}))$ are said to be equivalent if there exists invertible linear mapping $A:\mathcal{H}_{\phi}\rightarrow \mathcal{H}_{\psi}$ such that 
$$A\phi(x)=\psi(x) A$$
for all $x\in G.$ More detailed discussion of the above defined notions can be found in \cite{ruzhansky2010pseudo}.

\subsection{Noncommutative $\ell^{p}$-spaces associated with a compact group}

Let $G$ be a compact group. By $\widehat{G}$, we denote the unitary dual of $G$, i.e., the set of all equivalence classes of irreducible unitary representations from $\Rep(G)$ (see \cite[Definition 7.5.7 and 10.2.1]{ruzhansky2010pseudo}). Let $[\xi]\in\widehat{G}$ denote the equivalence class of a strongly continuous irreducible unitary representation $\xi:G\rightarrow \mathcal{U}(\mathcal{H}_{\xi}),$ where $\mathcal{H}_{\xi}$ is a representation space and note that $\mathcal{H}_{\xi}$ is finite dimensional since $G$ is compact. We also set $\dim(\xi)=\dim(\mathcal{H}_{\xi})$.

The space $\mathcal{M}(\widehat{G})$ consists of all mappings 
$$F:\widehat{G}\rightarrow \bigcup_{[\xi]\in\widehat{G}}\mathcal{L}(\mathcal{H}_{\xi})\subset\bigcup_{m=1}^{\infty}\mathbb{C}^{m\times m},$$
satisfying $F([\xi])\in\mathcal{L}(\mathcal{H}_{\xi})$ for every $[\xi]\in\widehat{G}$. Note that in matrix representations, one can view $F([\xi])\in\mathbb{C}^{\Dim(\xi)\times \Dim(\xi)}$ as a $\Dim(\xi)\times \Dim(\xi)$ matrix.

Let $\langle\xi \rangle:=\sqrt{1+\lambda_{[\xi]}^{2}},$ where $\lambda_{[\xi]}, [\xi]\in\G$, denotes the corresponding eigenvalue of the positive Laplacian (in a bijective manner) indexed by an equivalence class $[\xi]\in\G$ (for more details, see \cite[Definition 10.3.18]{ruzhansky2010pseudo}). The space $S'(\widehat{G})$ of slowly increasing or tempered distributions on the unitary dual $\widehat{G}$ is defined as the space of all $H\in\mathcal{M}(\widehat{G})$ for which there exists some $k\in\mathbb{N}$ such that
$$\sum_{[\xi]\in\widehat{G}}\Dim(\xi)\langle\xi\rangle^{-k}\lv H(\xi)\rv_{\mathcal{S}^{2}(\mathcal{H}_{\xi})}<\infty,$$
where $\lv\cdot\rv_{\mathcal{S}^{2}(\mathcal{H}_{\xi})}:=\lv \cdot\rv_{\mathcal{S}^{2}}$ is a Hilbert-Schmidt norm. The convergence in $S'(\widehat{G})$ is defined as follows: the sequence $H_{j}\in S'(\widehat{G})$ is said to be converging to $H\in S'(\widehat{G})$ in $S'(\widehat{G})$ as $j\rightarrow\infty$, if there exists some $k\in\mathbb{N}$ such that 
$$\sum_{[\xi]\in\widehat{G}}\Dim(\xi)\langle \xi\rangle^{-k}\lv H_{j}(\xi)-H(\xi)\rv_{\mathcal{S}^{2}(\mathcal{H}_{\xi})}\rightarrow0,\quad j\rightarrow\infty.$$

We now define $\ell^{p}$-spaces over the unitary dual $\widehat{G}$ of a compact group $G$ based on the Hilbert-Schmidt ideal.

\begin{definition}
\cite[Definition 10.3.36]{ruzhansky2010pseudo}
For $1\leq p<\infty$, the space $\ell^{p}(\widehat{G})\equiv \ell^{p}\left(\widehat{G},\Dim(\xi)^{p(\frac{2}{p}-\frac{1}{2}})\right)$ is given as the space of all $H\in S'(\widehat{G})$ such that
\begin{equation*}
\|H\|_{\ell^{p}(\widehat{G})}:=\left(\sum_{[\xi]\in\widehat{G}}(\Dim(\xi))^{p\left(\frac{2}{p}-\frac{1}{2}\right)}\lv H(\xi)\rv_{\mathcal{S}^{2}(\mathcal{H}_{\xi})}^{p}\right)^{1/p}<\infty.
\end{equation*}

For $p=\infty$, the space $\ell^{\infty}(\widehat{G})$ is given as the space of all $H\in S'(\widehat{G})$ such that 
\begin{equation*}
\|H\|_{\ell^{\infty}(\widehat{G})}:=\sup_{[\xi]\in\widehat{G}}(\Dim(\xi))^{-1/2}\lv H(\xi)\rv_{\mathcal{S}^{2}(\mathcal{H}_{\xi})}<\infty.
\end{equation*}
\end{definition}

We also present more general class of $\ell^p$-spaces over the unitary dual $\widehat{G}$ of a compact group $G$ based on the Schatten-von Neumann ideals. 

\begin{definition}\label{Schatten family of lp spaces}\cite[Section 2.1.4]{fischer2016quantization} (see also \cite[Section 2.14.2]{edwards1972integration})
For $1\leq p<\infty$, the space $\ell_{sch}^{p}(\widehat{G})$ is given as the space of all $H\in S'(\widehat{G})$ such that
\begin{equation*}
\|H\|_{\ell_{sch}^{p}(\widehat{G})}:=\left(\sum_{[\xi]\in\widehat{G}}\Dim(\xi)\lv H(\xi)\rv_{\mathcal{S}^{p}(\mathcal{H}_{\xi})}^{p}\right)^{1/p}<\infty,
\end{equation*} 
where $\lv\cdot\rv_{\mathcal{S}^{p}(\mathcal{H}_{\xi})}$ is the Schatten norm.

For $p=\infty$, the space $\ell_{sch}^{\infty}(\widehat{G})$ is given as the space of all $H\in S'(\widehat{G})$ such that 
\begin{equation*}
\|H\|_{\ell_{sch}^{\infty}(\widehat{G})}:=\sup_{[\xi]\in\widehat{G}}\lv H(\xi)\rv_{\mathcal{L}(\mathcal{H}_{\xi})}<\infty.
\end{equation*}    
\end{definition}

Note that the defined spaces $\left(\ell^{p}(\widehat{G}), \lv \cdot\prv\right)$ and $\left(\ell_{sch}^{p}(\widehat{G}), \lv \cdot\psrv\right)$ are Banach spaces for all $1\leq p\leq \infty$ (see, for example, \cite[Section 10.3.3]{ruzhansky2010pseudo} and \cite[Section 2.14.2]{edwards1972integration}, respectively).

There is a connection between $\ell^{p}(\widehat{G})$ and $\ell_{sch}^{p}(\widehat{G})$ for $1\leq p\leq \infty$ from \cite[Section 2.1.4]{fischer2016quantization}.

\begin{proposition}\cite[Proposition 2.1.6]{fischer2016quantization}
For $1\leq p\leq 2$, one has a continuous embedding $\ell^{p}(\widehat{G})\hookrightarrow \ell_{sch}^{p}(\widehat{G})$ and 
$$\lv H\psrv\leq \lv H\prv,\quad \forall H\in \ell^{p}(\widehat{G}).$$
For $2\leq p\leq \infty$, one has a continuous embedding $\ell_{sch}^{p}(\widehat{G})\hookrightarrow\ell^{p}(\G)$ and 
$$\lv H\prv\leq \lv H\psrv,\quad \forall H\in \ell_{sch}^{p}(\G).$$
\end{proposition}

We now present H\"{o}lder's inequality in $\ell_{sch}^{p}(\widehat{G})$ spaces. Since we could not find its proof elsewhere, we provide its proof for the convenience of the reader.
\begin{proposition}\label{Holder inequality}
Let $1\leq p,q,r\leq \infty$ such that $\frac{1}{r}=\frac{1}{p}+\frac{1}{q}.$ Then, for any $H_{1}\in\ell_{sch}^{p}(\G)$ and $H_{2}\in\ell_{sch}^{q}(\G),$ we have 
\begin{equation}\label{Holder}
\lv H_{1}H_{2}\rv_{\ell_{sch}^{r}(\G)}\leq \lv H_{1}\rv_{\ell_{sch}^{p}(\G)}\lv H_{2}\rv_{\ell_{sch}^{q}(\G)}.
\end{equation}
\end{proposition}

\begin{proof} Assume first that $r=\infty$, then it obviously follows that $p=q=\infty$. Since the uniform norm $\lv\cdot\rv_{\mathcal{L}(\mathcal{H}_{\xi})}$ is submultiplicative, it follows that
$$\lv H_{1}(\xi)H_{2}(\xi)\rv_{\mathcal{L}(\mathcal{H}_{\xi})}\leq \lv H_{1}(\xi)\rv_{\mathcal{L}(\mathcal{H}_{\xi})}\lv H_{2}(\xi)\rv_{\mathcal{L}(\mathcal{H}_{\xi})}\leq \sup_{[\xi]\in\G}\lv H_{1}(\xi)\rv_{\mathcal{L}(\mathcal{H}_{\xi})}\sup_{[\xi]\in\G}\lv H_{2}(\xi)\rv_{\mathcal{L}(\mathcal{H}_{\xi})}.$$
Hence, 
$$\lv H_{1}H_{2}\rv_{\ell_{sch}^{\infty}(\G)}=\sup_{[\xi]\in\G}\lv H_{1}(\xi)H_{2}(\xi)\rv_{\mathcal{L}(\mathcal{H}_{\xi})} \leq \sup_{[\xi]\in\G}\lv H_{1}(\xi)\rv_{\mathcal{L}(\mathcal{H}_{\xi})}\sup_{[\xi]\in\G}\lv H_{2}(\xi)\rv_{\mathcal{L}(\mathcal{H}_{\xi})}$$
$$=\lv H_{1}\rv_{\ell_{sch}^{\infty}(\G)}\lv H_{2}\rv_{\ell_{sch}^{\infty}(\G)}.$$

Suppose now that $r<\infty$, then at most one of the numbers $p$ and $q$ might equal to infinity. Without loss of generality assume that $p=\infty$ and $q=r.$ Recall that 
$$\lv AB \rv_{\mathcal{S}^{p}(\mathcal{H})}\leq \lv A\rv_{\mathcal{S}^{p}(\mathcal{H})}\lv B\rv_{\mathcal{L}(\mathcal{H})}$$
for any $A\in\mathcal{S}^{p}(\mathcal{H})$ and any bounded operator $B\in\mathcal{L}(\mathcal{H})$ (see, for example, \cite[Section II.2]{gohberg1969translations}). Hence, we have that
$$\lv H_{1}H_{2}\rv_{\ell_{sch}^r(\G)}=\left(\sum_{[\xi]\in\G}\Dim(\xi)\lv H_{1}(\xi)H_{2}(\xi)\rv_{\mathcal{S}^{r}(\mathcal{H}_{\xi})}^r\right)^{\frac{1}{r}}$$
$$\leq \left(\sum_{[\xi]\in\G}\Dim(\xi)\lv H_{1}(\xi)\rv_{\mathcal{L}(\mathcal{H}_{\xi})}^r\lv H_{2}(\xi)\rv_{\mathcal{S}^{r}(\mathcal{H}_{\xi})}^r\right)^{\frac{1}{r}}\leq \sup_{[\xi]\in\G}\lv H_{1}(\xi)\rv_{\mathcal{L}(\mathcal{H}_{\xi})}\left(\sum_{[\xi]\in\G}\Dim(\xi)\lv H_{2}(\xi)\rv_{\mathcal{S}^{r}(\mathcal{H}_{\xi})}^r\right)^{\frac{1}{r}}$$
$$=\lv H_{1}\rv_{\ell_{sch}^{\infty}(\G)}\lv H_{2}\rv_{\ell_{sch}^{r}(\G)}.$$

Finally it is left to consider the case $p,q,r<\infty.$ Note that, by H\"older's inequality in Schatten-von Neumann ideals, we have
\begin{equation*}\begin{split}
\sum_{[\xi]\in \G}\Dim(\xi)\lv (H_{1}H_{2})(\xi)\rv_{\mathcal{S}^{r}(\mathcal{H}_{\xi})}^r\leq \sum_{[\xi]\in\G}\Dim(\xi)^{r/p} \lv H_{1}(\xi)\rv_{\mathcal{S}^{p}(\mathcal{H}_{\xi})}^r\cdot \Dim(\xi)^{r/q}\lv H_{2}(\xi)\rv_{\mathcal{S}^q(\mathcal{H}_{\xi})}^r.
\end{split}\end{equation*}
Hence, by H\"older's inequality in classical $\ell^{p}$-spaces, it follows that 
$$\sum_{[\xi]\in\G}\Dim(\xi)^{r/p} \lv H_{1}(\xi)\rv_{\mathcal{S}^{p}(\mathcal{H}_{\xi})}^r\cdot \Dim(\xi)^{r/q}\lv H_{2}(\xi)\rv_{\mathcal{S}^q(\mathcal{H}_{\xi})}^r$$
$$\leq \left(\sum_{[\xi]\in\G}\left(\Dim(\xi)^{\frac{r}{p}}\right)^{\frac{p}{r}} \left(\lv H_{1}(\xi)\rv_{\mathcal{S}^{p}(\mathcal{H}_{\xi})}^r\right)^{\frac{p}{r}}\right)^{\frac{r}{p}}\cdot \left(\sum_{[\xi]\in\G}\left(\Dim(\xi)^{\frac{r}{q}}\right)^{\frac{q}{r}} \left(\lv H_{2}(\xi)\rv_{\mathcal{S}^{q}(\mathcal{H}_{\xi})}^r\right)^{\frac{q}{r}}\right)^{\frac{r}{q}}$$
$$=\left(\sum_{[\xi]\in\G}\Dim(\xi) \lv H_{1}(\xi)\rv_{\mathcal{S}^{p}(\mathcal{H}_{\xi})}^p\right)^{\frac{r}{p}}\cdot \left(\sum_{[\xi]\in\G}\Dim(\xi) \lv H_{2}(\xi)\rv_{\mathcal{S}^{q}(\mathcal{H}_{\xi})}^q\right)^{\frac{r}{q}}$$
$$=\lv H_{1}\rv_{\ell_{sch}^{p}(\G)}^{r}\lv H_{2}\rv_{\ell_{sch}^{q}(\G)}^r.$$
Therefore,
$$\sum_{[\xi]\in \G}\Dim(\xi)\lv (H_{1}H_{2})(\xi)\rv_{\mathcal{S}^{r}(\mathcal{H}_{\xi})}^r\leq \lv H_{1}\rv_{\ell_{sch}^{p}(\G)}^{r}\lv H_{2}\rv_{\ell_{sch}^{q}(\G)}^{r}<\infty,$$
and $H_{1}H_{2}\in\ell_{sch}^{r}(\G).$
This means 
$$\lv H_{1}H_{2}\rv_{\ell_{sch}^{r}(\G)}\leq \lv H_{1}\rv_{\ell_{sch}^{p}(\G)}\lv H_{2}\rv_{\ell_{sch}^{q}(\G)},$$
completing the proof.
\end{proof}

Let $H\in \ell^{p}(\widehat{G})$ or $H\in \ell_{sch}^{p}(\widehat{G})$. By $H^{\ast}$, we denote the adjoint element of $H$, which is defined as 
$$\left(H^{\ast}\right)(\xi)=\left(H(\xi)\right)^{\ast},\quad [\xi]\in\widehat{G}.$$ 
Similarly, we define its absolute value $|H|$ as follows
$$(|H|)(\xi)=\left(H^{\ast}(\xi)H(\xi)\right)^{\frac{1}{2}},\quad [\xi]\in\G.$$
Note that $H^{\ast}$ and $|H|$ are also the mappings from $\widehat{G}$ into $\bigcup\limits_{\xi\in\widehat{G}}\mathcal{L}(\mathcal{H}_{\xi}).$

Part (ii) of the following fact can be found in \cite[p. 146]{edwards1972integration}, however, without proof. Hence, for the convenience of the reader, we include its proof.

\begin{proposition}\label{adjoint}

\begin{enumerate}[label=(\roman*)]
\item Let $H\in \ell^{p}(\widehat{G}).$ Then, $H^{\ast}\in \ell^{p}(\widehat{G})$ and $\lv H^{\ast}\prv=\lv H\prv=\lv \lvert H\rvert \prv.$
\item Let $H\in \ell_{sch}^{p}(\widehat{G}).$ Then, $H^{\ast}\in \ell_{sch}^{p}(\widehat{G})$ and $\lv H^{\ast}\psrv=\lv H\psrv=\lv \lvert H\rvert \psrv.$
\end{enumerate}
\end{proposition}

\begin{proof}
(i). Note that since the sequence of singular values of the operators $H(\xi)$, $H^{\ast}(\xi)$ and $\lvert H(\xi)\rvert$ coincide for all $[\xi]\in\widehat{G}$, it follows that $\lv H(\xi)\rv_{\mathcal{S}^{2}(\mathcal{H}_{\xi})}=\lv H^{\ast}(\xi)\rv_{\mathcal{S}^{2}(\mathcal{H}_{\xi})}=\lv \lvert H(\xi)\lvert\rv_{\mathcal{S}^{2}(\mathcal{H}_{\xi})},$ $[\xi]\in\widehat{G},$ where $\lv\cdot\rv_{\mathcal{S}^{2}(\mathcal{H}_{\xi})}$ is the Hilbert-Schmidt norm (see, for example, \cite[Sections II.2 and III.2]{gohberg1969translations}). 
Therefore, for $1\leq p<\infty$, one has
$$\lv H^{\ast}\prv^{p}=\sum_{[\xi]\in\widehat{G}}\left(\Dim(\xi)\right)^{p\left(\frac{2}{p}-\frac{1}{2}\right)}\lv \left(H(\xi)\right)^{\ast}\rv_{\mathcal{S}^{2}(\mathcal{H}_{\xi})}=\sum_{[\xi]\in\widehat{G}}\left(\Dim(\xi)\right)^{p\left(\frac{2}{p}-\frac{1}{2}\right)}\lv H(\xi)\rv_{\mathcal{S}^{2}(\mathcal{H}_{\xi})}=\lv H\prv^{p}$$
$$=\sum_{[\xi]\in\widehat{G}}\left(\Dim(\xi)\right)^{p\left(\frac{2}{p}-\frac{1}{2}\right)}\lv \lvert H(\xi)\rvert\rv_{\mathcal{S}^{2}(\mathcal{H}_{\xi})}=\lv \lvert H
\rvert\prv^{p}.$$
Similarly, for $p=\infty$, one has 
$$\|H^{\ast}\|_{\ell^{\infty}(\widehat{G})}=\sup_{[\xi]\in\widehat{G}}(\Dim(\xi))^{-1/2}\lv \left(H(\xi)\right)^{\ast}\rv_{\mathcal{S}^{2}(\mathcal{H}_{\xi})}=\sup_{[\xi]\in\widehat{G}}(\Dim(\xi))^{-1/2}\lv H(\xi)\rv_{\mathcal{S}^{2}(\mathcal{H}_{\xi})}=\lv H\rv_{\ell^{\infty}(\widehat{G})}$$
$$=\sup_{[\xi]\in\widehat{G}}(\Dim(\xi))^{-1/2}\lv \lvert H(\xi)\rvert\rv_{\mathcal{S}^{2}(\mathcal{H}_{\xi})}=\lv \lvert H\rvert\rv_{\ell^{\infty}(\widehat{G})},$$
which together with the last equality complete the proof.

(ii). Similarly, since the singular values of the operators $H(\xi),$ $H^\ast(\xi)$ and $\lvert H(\xi)\rvert$ coincide for all $[\xi]\in\widehat{G}$, it follows that $\lv H(\xi)\rv_{\mathcal{S}^{p}(\mathcal{H}_{\xi})}=\lv H^{\ast}(\xi)\rv_{\mathcal{S}^{p}(\mathcal{H}_{\xi})}=\lv \lvert H(\xi)\rvert\rv_{\mathcal{S}^{p}(\mathcal{H}_{\xi})},$ $[\xi]\in\widehat{G},$ where $\lv\cdot\rv_{\mathcal{S}^{p}(\mathcal{H}_{\xi})}$ is the Schatten norm (see, for example, \cite[Sections II.2 and III.2]{gohberg1969translations}). 
Therefore, for $1\leq p<\infty$, one has
$$\lv H^{\ast}\psrv^{p}=\sum_{[\xi]\in\widehat{G}}\Dim(\xi)\lv \left(H(\xi)\right)^{\ast}\rv_{\mathcal{S}^{p}(\mathcal{H}_{\xi})}=\sum_{[\xi]\in\widehat{G}}\Dim(\xi)\lv H(\xi)\rv_{\mathcal{S}^{p}(\mathcal{H}_{\xi})}=\lv H\psrv^{p}$$
$$=\sum_{[\xi]\in\widehat{G}}\Dim(\xi)\lv \lvert H(\xi)\rvert\rv_{\mathcal{S}^{p}(\mathcal{H}_{\xi})}=\lv \lvert H\rvert\psrv^{p}.$$
Similarly, for $p=\infty$, one has 
$$\|H^{\ast}\|_{\ell_{sch}^{\infty}(\widehat{G})}=\sup_{[\xi]\in\widehat{G}}\lv \left(H(\xi)\right)^{\ast}\rv_{\mathcal{L}(\mathcal{H}_{\xi})}=\sup_{[\xi]\in\widehat{G}}\lv H(\xi)\rv_{\mathcal{L}(\mathcal{H}_{\xi})}=\lv H\rv_{\ell_{sch}^{\infty}(\widehat{G})}$$
$$=\sup_{[\xi]\in\widehat{G}}\lv \lvert H(\xi)\rvert\rv_{\mathcal{L}(\mathcal{H}_{\xi})}=\lv \lvert H\rvert\rv_{\ell_{sch}^{\infty}(\widehat{G})},$$
which together with the last equality complete the proof.
\end{proof}

For the noncommutative spaces $\ell^{p}(\widehat{G})$, $1<p<\infty$, the following Clarkson type inequalities are known from \cite[Theorem 3]{tulenov2018clarkson}. The proof can also be found in \cite{tulenov2018clarkson}.

\begin{proposition}\label{Clarkson type inequality}
Let $1<p,q<\infty$ with $\frac{1}{p}+\frac{1}{q}=1$. Then, for any $H_{1}, H_{2}\in \ell^{p}(\widehat{G})$, one has the following inequalities:
\begin{enumerate}[label=(\roman*)]
    \item 
    If $1<p\leq 2$, then 
    \begin{equation*}\begin{split}
    \Biggl(\lv \frac{H_{1}+H_{2}}{2} \prv^{q}+\lv \frac{H_{1}-H_{2}}{2} \prv^{q}\Biggr)^{1/q}\leq \left(\frac{1}{2}\left(\lv H_{1}\prv^{p}+\lv H_{2}\prv^{p}\right)\right)^{1/p};
    \end{split}\end{equation*}
    \item 
    If $2\leq p< \infty$, then \begin{equation*}\begin{split}
    \Biggl(\lv \frac{H_{1}+H_{2}}{2} \prv^{p}+\lv \frac{H_{1}-H_{2}}{2} \prv^{p}\Biggr)^{1/p}\leq \left(\frac{1}{2}\left(\lv H_{1}\prv^{q}+\lv H_{2}\prv^{q}\right)\right)^{1/q}.
    \end{split}\end{equation*}
\end{enumerate}
\end{proposition}

\subsection{Uniformly convex and uniformly smooth Banach spaces, type and co-type properties}\label{section type}

In general, for a given Banach space, one can define the notions of its modulus of convexity and modulus of smoothness.

\begin{definition}\label{modulus}
Let $(X,\|\cdot\|_{X})$ be a Banach space. Its modulus of convexity and modulus of smoothness are defined by 
$$\delta_{X}(\varepsilon):=\inf\left\{1-\lv\frac{x+y}{2}\rv_{X} \mid x,y\in X,\quad \lv x\rv_{X}=\lv y\rv_{X}=1,\quad \lv x-y\rv_{X}=\varepsilon\right\},$$ for $0<\varepsilon\leq2$, and 
$$\rho_{X}(t):=\sup\left\{\frac{\lv x+ty\rv_{X}+\lv x-ty\rv_{X}}{2}-1 \mid x,y\in X,\quad \lv x\rv_{X}=\lv y\rv_{X}=1\right\},$$
for $t>0,$ respectively.
\end{definition}

These notions are helpful to classify uniformly convex and uniformly smooth Banach spaces, respectively. 

\begin{definition}\label{uniform convexity and smoothness}
Let $X$ be a Banach space, and $\delta_{X}(\varepsilon),$ $0<\varepsilon<2,$ and $\rho_{X}(t),$ $t>0,$ be its modulus of convexity and modulus of smoothness, respectively. Then, the Banach space $X$ is said to be uniformly convex if $\delta_{X}(\varepsilon)>0$ for every $\varepsilon>0,$ and uniformly smooth if $\lim\limits_{t\rightarrow 0} \frac{\rho_{X}(t)}{t}=0.$
\end{definition}

Note that the notions of uniform convexity and uniform smoothness are the dual notions to each other. Namely, in \cite[Propositition 1.e.2]{lindenstrauss2013classical} (see also \cite{day1944uniform, lindenstrauss1963modulus}), it was proved that a Banach space $X$ is uniformly convex if and only if $X^{\ast}$ is uniformly smooth, where $X^\ast$ is the dual space. Moreover, in the same proposition, authors proved the identity connecting the modulus of convexity of $X$ and the modulus of smoothness of its dual $X^\ast$.

Note that the simplest example of a Banach space that is both uniformly convex and uniformly smooth is a Hilbert space. Moreover, in this case, the modulus of convexity and modulus of smoothness of a given Banach space $H$ can be computed easily (due to the parallelogram identity) as follows
$$\delta_{H}(\varepsilon)=1-\left(1-\frac{\varepsilon^2}{4}\right)^{\frac{1}{2}}=\frac{\varepsilon^2}{8}+O(\varepsilon^4),\quad 0<\varepsilon<2,$$
$$\rho_{H}(t)=\left(1+t^2\right)^{\frac{1}{2}}=\frac{t^2}{2}+O(t^4),\quad t>0.$$

It is known due to Nordlander \cite{Nordlander1960} (see also \cite{lindenstrauss2013classical}) that a Hilbert space is the \enquote{most} uniformly convex and the \enquote{most} uniformly smooth space among all Banach spaces, i.e., one has
\begin{equation*}\begin{split}
\delta_{X}(\varepsilon)&\leq \delta_{H}(\varepsilon),\quad 0<\varepsilon<2,\\
\quad\rho_{X}(t)&\geq \rho_{H}(t),\quad t>0.
\end{split}\end{equation*}
Hence, for any Banach space $X$, there exist constants $C, S>0$ such that
\begin{equation}\begin{split}\label{optimal rate in general}
\delta_{X}(\varepsilon)&\leq C\varepsilon^2,\quad 0<\varepsilon<2,\\
\rho_{X}(t)&\geq St^2,\quad t>0.
\end{split}\end{equation}

We now present one of the interesting applications of these notions connecting them to the reflexivity of the Banach space. Recall that a Banach space $X$ is said to be reflexive if $X^{\ast\ast}=X$, where $X^{\ast\ast}$ denotes the second dual of $X.$ The proof of the following result can be found in \cite[Proposition 1.e.3]{lindenstrauss2013classical} (see also \cite{milman1938some, pettis1939proof}). 
\begin{proposition}\label{MilmanPettis}
Every uniformly convex Banach space is reflexive. Thus, every uniformly smooth Banach space is reflexive. 
\end{proposition}

We now recall the definition of the type and cotype of the Banach space $X.$ Let $(X,\lv \cdot\rv)$ be a Banach space and $x_{j}\in X,$ $1\leq j\leq n$. Consider the following expression 
\begin{equation*}\begin{split}
\underset{\theta_{j}=\pm1}{\rm{Average}}\lv\sum_{j=1}^{n} \theta_{j}x_{j}\rv=\frac{1}{2^n}\sum_{\theta_{j}=\pm1}\lv \sum_{j=1}^{n}\theta_{j}x_{j}\rv=\int_{0}^{1}\lv \sum_{j=1}^{n}r_{j}(t)x_{j} \rv dt,
\end{split}\end{equation*}
where $\{r_{j}\}_{j=1}^{\infty}$ denotes the sequence of Rademacher functions (see \cite[Section 1.e]{lindenstrauss2013classical} or \cite[Section 6.2]{albiac2006topics}  for more details), i.e. $$r_{j}(t)=\mathrm{sgn}\left(\sin 2^{j}\pi t\right),\quad j\in\mathbb{N},\quad t\in[0,1].$$ 
Note that first two expressions are to be understood as the mean value over all possible sums with either $\theta_{j}=1$ or $\theta_{j}=-1.$ For example, it equals to 
$$\frac{\lv x_{1}+x_{2}\rv+\lv x_{1}-x_{2}\rv}{2}$$
when $n=2.$

\begin{definition}\label{type def}\cite[Definition 1.e.12]{lindenstrauss2013classical}
A Banach space $X$ is said to be type $p$ for some $1<p\leq 2$ if there exists  constant $M$ such that for every finite number of vectors $\{x_{j}\}_{j=1}^{n}$ in $X$, we have 
\begin{equation*}
\int_{0}^{1}\lv \sum_{j=1}^{n}r_{j}(t)x_{j} \rv dt\leq M\left(\sum_{j=1}^{n}\lv x_{j}\rv^{p}\right)^{1/p},
\end{equation*}
and $X$ is said to be cotype $q$ for some $q\geq 2$ if there exists constant $M$ such that for every finite number of vectors $\{x_{j}\}_{j=1}^{n}$ in $X$, we have
\begin{equation*}
\int_{0}^{1}\lv \sum_{j=1}^{n}r_{j}(t)x_{j} \rv dt\geq M\left(\sum_{j=1}^{n}\lv x_{j}\rv^{q}\right)^{1/q}.
\end{equation*}
\end{definition}

Note that every Banach space is of type $1$ and cotype $\infty$ (see, for example, \cite[Section 1.e, p. 73]{lindenstrauss2013classical}). Hence, these trivial cases are not included in the definition of type and cotype of the Banach space. Also note that assumptions on $p$ and $q$ in the above definition are not restrictive, since by Khintchine's inequality (see, for example, \cite[Theorem 2.b.3]{lindenstrauss2013classicalsequence} or \cite[Theorem 6.2.3]{albiac2006topics}) it follows that any Banach space cannot be of type $p>2$ and cotype $1\leq q< 2$. Every Hilbert space is of type $2$ and cotype $2.$ Conversely, if a Banach space is of type $2$ and cotype $2$, then it is isomorphic to a Hilbert space (see, \cite[Proposition 3.1]{kwapien1972isomorphic} or \cite{lindenstrauss2013classical} for more details).

\begin{remark}\label{different averages of type}
The $L^{1}$-average $\int_{0}^{1}\lv \sum\limits_{j=1}^{n}r_{j}(t)x_{j}\rv dt$ in Definition \ref{type def} can equivalently be replaced by any other $L^{r}$-average with $1<r<\infty$ due to the Kahane inequality (see \cite[Theorem 1.e.13]{lindenstrauss2013classical}) which is stated as follows: For every $1<r<\infty$, there exists constant $0<K_{r}<\infty$ such that for every finite sequence $\{x_{j}\}_{j=1}^{n}\subset X$, we have 
\begin{equation*}
    \int_{0}^{1}\lv\sum_{j=1}^{n}r_{j}(t)x_{j}\rv dt\leq \left(\int_{0}^{1}\lv \sum_{j=1}^{n}r_{j}(t)x_{j}\rv^{r}dt\right)^\frac{1}{r}\leq K_{r}\int_{0}^{1}\lv \sum_{j=1}^{n}r_{j}(t)x_{j}\rv dt.
\end{equation*}
Especially, we will be interested in the $L^2$-average $\left(\int_{0}^{1}\lv \sum\limits_{j=1}^{n}r_{j}(t)x_{j}\rv^2dt\right)^{1/2} $ in the next section.
\end{remark}

\subsection{Complex interpolation}

Let $(X_{0}, \lv \cdot\rv_{X_{0}})$ and $(X_{1}, \lv \cdot\rv_{X_{1}})$ be (complex) Banach spaces. One can define their intersection $X_{0}\cap X_{1}$ and sum $X_{0}+X_{1}$ equipped with the norms 
$$\lv x\rv_{X_{0}\cap X_{1}}:=\max\{\lv x\rv_{X_{0}}, \lv x\rv_{X_{1}}\},\quad x\in X_{0}\cap X_{1}$$
and 
$$\lv x\rv_{X_{0}+X_{1}}:=\inf\{\lv x_{0}\rv_{X_{0}}+\lv x_{1}\rv_{X_{1}}:\quad x=x_{0}+x_{1},\quad x_{0}\in X_{0}, \quad x_{1}\in X_{1}\},$$
respectively. Obviously, $(X_{0}\cap X_{1}, \lv \cdot\rv_{X_{0}\cap X_{1}})$ and $(X_{0}+X_{1}, \lv \cdot\rv_{X_{0}+X_{1}})$ are Banach spaces too. 

Consider the space $\mathcal{F}:=\mathcal{F}(X_{0}, X_{1})$ of all functions $f$ with values in $X_{0}+X_{1}$, which are bounded and continuous on the strip $\{z\in\mathbb{C}: 0\leq \mathrm{Re}z\leq 1\}$ and analytic on the open strip $\{z\in\mathbb{C}: 0< \mathrm{Re}z< 1\}$, and the functions $t\mapsto f(j+it),$ $j=0,1$ are continuous from the real line into $X_{j}$, which tends to zero as $|t|\rightarrow\infty.$ This space becomes a Banach space equipped with a norm 
$$\lv f\rv_{\mathcal{F}}=\max\{\sup_{t\in\mathbb{R}}\lv f(it)\rv_{X_{0}}, \sup_{t\in\mathbb{R}}\lv f(1+it)\rv_{X_{1}}\}.$$

For $0<\theta<1$, a complex interpolation space $(X_{0}, X_{1})_{\theta}$ is defined as the class of all $x\in X_{0}+X_{1}$ such that $x=f(\theta)$ for some $f\in \mathcal{F}(X_{0}, X_{1})$ equipped with a norm 
$$\lv x\rv_{(X_{0}, X_{1})_{\theta}}=\inf\{\lv f\rv_{\mathcal{F}}: \quad f(\theta)=x,\quad f\in \mathcal{F}(X_{0}, X_{1})\}.$$ Note that $\left((X_{0}, X_{1})_{\theta}, \lv \cdot\rv_{(X_{0}, X_{1})_{\theta}}\right)$ is a Banach space such that $$X_{0}\cap X_{1}\subset (X_{0}, X_{1})_{\theta}\subset X_{0}+X_{1}$$
with continuous inclusions.

There is another well-known complex interpolation space $(X_{0}, X_{1})^{\theta}$ (see, for example, \cite[Section 4.1]{lofstrom1976interpolation}). However, since the spaces that are considered in this paper are all reflexive (cf. Proposition \ref{reflexivity}), these two complex interpolations coincide (see \cite[Theorem 4.3.1]{lofstrom1976interpolation}). Hence, throughout this paper, we denote by $(X_{0},X_{1})_{\theta}$ the complex interpolation space of the (complex) Banach spaces $X_{0}$ and $X_{1}$, which is defined as above (for more detailes, we refer the reader to \cite{lofstrom1976interpolation, calderon1964intermediate}).

Let $(X, \lv\cdot\rv_{X})$ be a Banach space and $1\leq p\leq \infty.$ A space $X\oplus_{p}X$ is the set of all pairs $(x,y)\in X\times X$ equipped with the norm
$$\lv (x,y)\rv_{X\oplus_{p}X}:=\left(\lv x\rv_{X}^{p}+\lv y\rv_{X}^{p}\right)^{1/p}$$
when $p<\infty$, and 
$$\lv (x,y)\rv_{X\oplus_{\infty}X}:=\max\{\lv x\rv_{X}, \lv y\rv_{X}\}$$
when $p=\infty.$

The following interpolation result due to Calder\'on is important to prove Clarkson type inequalities in later section. 

\begin{proposition}\label{calderon interpolation 1}\cite[Section 13.6]{calderon1964intermediate} (see also \cite{pisier1998non})
Let $X_{1}, X_{2}$ be Banach spaces and $1\leq p_{1},p_{2}\leq \infty$. If at least one of the spaces $X_{1}\oplus_{p_{1}}X_{1},$ $X_{2}\oplus_{p_{2}}X_{2}$ is reflexive, then
$$\left(X_{1}\oplus_{p_{1}}X_{1}, X_{2}\oplus_{p_{2}}X_{2}\right)_{\theta}=\left(X_{1},X_{2}\right)_{\theta}\oplus_{p}\left(X_{1},X_{2}\right)_{\theta},$$
where $\frac{1}{p}=\frac{1-\theta}{p_{1}}+\frac{\theta}{p_{2}},$ $0<\theta<1.$
\end{proposition}

Let $w>0$. A space $X\oplus_{p,w}X$ is the set of all pairs $(x,y)\in X\times X$ equipped with the norm 
$$\lv (x,y)\rv_{X\oplus_{p,w}X}:=\left(\lv x\rv_{X}^{p}+w\lv y\rv_{X}^{p}\right)^{1/p}$$
when $p<\infty$, and 
$$\lv (x,y)\rv_{X\oplus_{\infty,w}X}:=\max\{\lv x\rv_{X}, w \lv y\rv_{X}\}$$
when $p=\infty.$

We now recall the modified version of Proposition \ref{calderon interpolation 1}.

\begin{proposition}\label{calderon interpolation 2}\cite[Section 13.6]{calderon1964intermediate} (see also \cite[Theorem 5.5.3]{lofstrom1976interpolation} and \cite{pisier1998non}) Let $X_{1}, X_{2}$ be Banach spaces and $1\leq p_{1}, p_{2}\leq \infty$. If $w_{1}, w_{2}>0$ and at least one of the spaces $X_{1}\oplus_{p_{1},w_{1}}X_{1}$, $X_{2}\oplus_{p_{2},w_{2}}X_{2}$ is reflexive, then 
$$\left(X_{1}\oplus_{p_{1},w_{1}}X_{1}, X_{2}\oplus_{p_{2},w_{2}}X_{2}\right)_{\theta}=\left(X_{1},X_{2}\right)_{\theta}\oplus_{p,w}\left(X_{1},X_{2}\right)_{\theta},$$
where $\frac{1}{p}=\frac{1-\theta}{p_{1}}+\frac{\theta}{p_{2}},$ $0<\theta<1$ and $w=w_{1}^{p(1-\theta)/p_{1}}w_{2}^{p\theta/p_{2}}.$    
\end{proposition}

\section{Family of $\ell^{p}$-spaces $\ell_{sch}^{p}(\widehat{G})$ based on Schatten-von Neumann ideals}

\subsection{Duality results}\label{subsectiondual}

We first present a reflexivity result of the spaces $\ell_{sch}^{p}(\G)$, $1<p<\infty$, which will be useful to prove the duality of these spaces. 
\begin{proposition}\label{reflexivity}
The space $\ell_{sch}^{p}(\G),$ $1<p<\infty,$ is reflexive, i.e., $\left(\ell_{sch}^{p}(\G)\right)^{\ast\ast}\cong\ell_{sch}^{p}(\G).$  
\end{proposition}
\begin{proof} Note that $\ell_{sch}^{p}(\G),$ $1<p<\infty,$ is uniformly convex, by Theorem \ref{mainth0}, which will be proved in Subsection \ref{geometric properties}. We end the proof with the Milman-Pettis theorem which states that every uniformly convex space is reflexive.  
\end{proof}

We now prove the following duality between spaces $\ell_{sch}^{p}(\G)$.
\begin{proposition}\label{duality of lp}
Let $1\leq p<\infty$ and $\frac{1}{p}+\frac{1}{q}=1.$ Then
\begin{enumerate}[label=(\roman*)]
\item\label{duality statement 1} For any $H\in\ell_{sch}^{p}(\G)$, we have
\begin{equation}\label{equality for duality bracket}
\lv H\psrv=\sup\{\left|\langle H, F\rangle_{\G}\right|:\quad F\in\ell_{sch}^{q}(\G),\quad \lv F\rv_{\ell_{sch}^{q}(\G)}=1\},
\end{equation}
where $$\langle H, F\rangle_{\G}=\sum_{[\xi]\in\G}\Dim(\xi)\mathrm{Tr}(H(\xi)F(\xi)).$$
\item\label{duality statement 2} $\left(\ell_{sch}^{p}(\G)\right)^{\ast}$ is isometrially isomorphic to $\ell_{sch}^{q}(\G)$, i.e. $\left(\ell_{sch}^{p}(\G)\right)^{\ast}\cong \ell_{sch}^{q}(\G).$
\end{enumerate}
\end{proposition}
\begin{proof} 
\ref{duality statement 1}. Let $1< p<\infty$, and $H\in\ell_{sch}^{p}(\G)$ and $F\in\ell_{sch}^{q}(\G).$ Then, by \cite[Theorem III.8.5, p. 104]{gohberg1969translations}, one has
\begin{equation*}\begin{split}
\lvert \langle H,F\rangle_{\G} \rvert\leq \sum_{[\xi]\in\G}\Dim(\xi) \lv H(\xi)F(\xi)\rv_{\mathcal{S}^{1}(\mathcal{H}_{\xi})}.
\end{split}\end{equation*}
Hence, 
\begin{equation*}\begin{split}
&\sum_{[\xi]\in\G}\Dim(\xi) \lv H(\xi)F(\xi)\rv_{\mathcal{S}^{1}(\mathcal{H}_{\xi})}\overset{\eqref{Holder}}{\le}\sum_{[\xi]\in\G}\Dim(\xi)^{\frac{1}{p}} \lv H(\xi)\rv_{\mathcal{S}^{p}(\mathcal{H}_{\xi})}\Dim(\xi)^{\frac{1}{q}}\lv F(\xi)\rv_{\mathcal{S}^{q}(\mathcal{H}_{\xi})}\\
&\leq \left[\sum_{[\xi]\in\G}\left(\Dim(\xi)^{\frac{1}{p}}\right)^{p} \lv H(\xi)\rv_{\mathcal{S}^{p}(\mathcal{H}_{\xi})}^{p}\right]^{\frac{1}{p}} \left[\sum_{[\xi]\in\G}\left(\Dim(\xi)^{\frac{1}{q}}\right)^{q} \lv F(\xi)\rv_{\mathcal{S}^{q}(\mathcal{H}_{\xi})}^{q}\right]^{\frac{1}{q}}=\lv H\rv_{\ell_{sch}^{p}(\G)}\lv F\rv_{\ell_{sch}^{q}(\G)}.
\end{split}\end{equation*}
If $p=1$, then by \cite[Theorem III.8.5, p. 104]{gohberg1969translations}, one has
\begin{equation*}\begin{split}
\lvert \langle H,F\rangle_{\G}& \rvert\leq \sum_{[\xi]\in \G}\Dim(\xi)\lv H(\xi) F(\xi)\rv_{\mathcal{S}^{1}(\mathcal{H}_{\xi})}\leq \sum_{[\xi]\in \G}\Dim(\xi)\lv H(\xi)\rv_{\mathcal{S}^{1}(\mathcal{H}_{\xi})}\lv F(\xi)\rv_{\mathcal{L}(\mathcal{H}_{\xi})}\\
&\leq \sum_{[\xi]\in \G}\Dim(\xi)\lv H(\xi)\rv_{\mathcal{S}^{1}(\mathcal{H}_{\xi})} \sup_{[\xi]\in\G}\lv F(\xi)\rv_{\mathcal{L}(\mathcal{H}_{\xi})}\leq \sum_{[\xi]\in \G}\Dim(\xi)\lv H(\xi)\rv_{\mathcal{S}^{1}(\mathcal{H}_{\xi})} \lv H\rv_{\ell_{sch}^{\infty}(\G)}\\
&=\lv H\rv_{\ell_{sch}^{1}(\G)}\lv F\rv_{\ell_{sch}^{\infty}(\G)}.
\end{split}\end{equation*}
Therefore, for all $1\leq p<\infty$, we have
$$\lvert \langle H,F\rangle_{\G} \rvert\leq\lv H\rv_{\ell_{sch}^{p}(\G)}\lv F\rv_{\ell_{sch}^{q}(\G)}$$
with $\frac{1}{p}+\frac{1}{q}=1.$
Thus, we have
\begin{equation}\label{duality bracket inequality 1}
\lv H\psrv\geq\sup\{\left|\langle H, F\rangle_{\G}\right|:\quad F\in\ell_{sch}^{q}(\G),\quad \lv F\rv_{\ell_{sch}^{q}(\G)}=1\}.
\end{equation}

For $H\in\ell_{sch}^{p}(\G)$, construct $F\in\ell_{sch}^{q}(\G)$ by the equality 
$$F(\xi)=\frac{\lvert H(\xi)\rvert^{p-1}U^{\ast}(\xi)}{\lv H\psrv^{p-1}},\quad [\xi]\in\G,$$
where $U(\xi)$ is the partial isometry from the polar decomposition of the finite-rank operators $H(\xi)=U(\xi)\lvert H(\xi)\rvert,$ $[
\xi]\in\G.$ Note that
\begin{equation*}\begin{split}
\lv F\rv_{\ell_{sch}^{q}(\G)}^{q}&=\sum_{[\xi]\in\G}\Dim(\xi)\lv F(\xi)\rv_{\mathcal{S}^{q}(\mathcal{H}_{\xi})}^{q}=\sum_{[\xi]\in\G}\Dim(\xi)\mathrm{Tr}\left(\lvert F^{\ast}(\xi)\rvert^{q}\right)\\
&=\sum_{[\xi]\in\G}\Dim(\xi)\mathrm{Tr}\left((F(\xi)F^{\ast}(\xi))^{\frac{q}{2}}\right)=\sum_{[\xi]\in\G}\Dim(\xi)\mathrm{Tr}\left(\left(\frac{\lvert H(\xi)\rvert^{p-1}U^{\ast}(\xi)}{\lv H\psrv^{p-1}}\cdot\frac{U(\xi)\lvert H(\xi)\rvert^{p-1}}{\lv H\psrv^{p-1}}\right)^{\frac{q}{2}}\right)\\
&=\sum_{[\xi]\in\G}\Dim(\xi)\mathrm{Tr}\left(\left(\frac{\lvert H(\xi)\rvert^{2(p-1)}}{\lv H\psrv^{2(p-1)}}\right)^{\frac{q}{2}}\right)=\frac{1}{\lv H\psrv^{p}}\sum_{[\xi]\in\G}\Dim(\xi)\mathrm{Tr}\left(\lvert H(\xi)\rvert^{p}\right)\\
&=\frac{1}{\lv H\psrv^{p}}\sum_{[\xi]\in\G}\Dim(\xi)\lv H(\xi)\rv_{\mathcal{S}^{p}(\mathcal{H}_{\xi})}^{p}=\frac{\lv H\psrv^{p}}{\lv H\psrv^{p}}=1.
\end{split}\end{equation*}
Moreover, 
\begin{equation}\begin{split}\label{duality unitar inv}
\langle H, F\rangle_{\G}&=\sum_{[\xi]\in\G}\Dim(\xi)\mathrm{Tr}(H(\xi)F(\xi))=\sum_{[\xi]\in\G}\Dim(\xi)\mathrm{Tr}\left(\frac{H(\xi)\lvert H(\xi)\rvert^{p-1}U^{\ast}(\xi)}{\lv H\psrv^{p-1}}\right)\\
&=\sum_{[\xi]\in\G}\Dim(\xi)\mathrm{Tr}\left(\frac{U(\xi)\lvert H(\xi)\rvert^{p}U^{\ast}(\xi)}{\lv H\psrv^{p-1}}\right)=\sum_{[\xi]\in\G}\Dim(\xi)\mathrm{Tr}\left(\frac{\lvert H(\xi)\rvert^{p}}{\lv H\psrv^{p-1}}\right)\\
&=\frac{1}{\lv H\psrv^{p-1}}\sum_{[\xi]\in\G}\Dim(\xi)\mathrm{Tr}\left(\lvert H(\xi)\rvert^{p}\right)=\frac{1}{\lv H\psrv^{p-1}}\sum_{[\xi]\in\G}\Dim(\xi)\lv H(\xi)\rv_{\mathcal{S}^{p}(\mathcal{H}_{\xi})}^{p}\\
&=\frac{\lv H\psrv^{p}}{\lv H\psrv^{p-1}}=\lv H\psrv,
\end{split}\end{equation}
where the fourth equality follows from the fact that the trace is unitarily invariant. Hence, \eqref{duality unitar inv} together with \eqref{duality bracket inequality 1} end the proof of \ref{duality statement 1}.

\ref{duality statement 2}. Let $F\in \ell_{sch}^{q}(\widehat{G})$. Then define a mapping $\Phi:\ell_{sch}^{q}(\widehat{G})\mapsto\left(\ell_{sch}^{p}(\G)\right)^{\ast}$ in the following form (see \cite[Definition 10.3.29]{ruzhansky2010pseudo} for more details)
\begin{equation*}
\Phi(F)(H)=\langle F,H \rangle_{\G}:=\sum_{[\xi]\in\G}\Dim(\xi)\mathrm{Tr}(F(\xi)H(\xi)),\quad H\in\ell_{sch}^{p}(\G).
\end{equation*}

We first show that the mapping $\Phi$ is norm preserving. By part \ref{duality statement 1}, one has
\begin{equation}\begin{split}
\lv \Phi(F)\rv_{\left(\ell_{sch}^{p}(\G)\right)^{\ast}}&=\sup\left\{\left|\Phi(F)(H)\right|:\quad H\in\ell_{sch}^{p}(\G),\quad \lv H\psrv=1
\right\}\\
&=\sup\left\{\left|\langle F, H\rangle_{\G}\right|:\quad H\in\ell_{sch}^{p}(\G),\quad \lv H\psrv=1
\right\}=\lv F\rv_{\ell_{sch}^{q}(\G)},\quad \forall F\in\ell_{sch}^{q}(\G).
\end{split}\end{equation}

Finally, it is left to show that the mapping $\Phi$ is surjective. Since $\ell_{sch}^{q}(\G)$ is Banach space and $\Phi$ is bounded linear mapping, it follows that $\Phi(\ell_{sch}^{q}(\G))$ is a closed subspace of $\left(\ell_{sch}^{p}(\G)\right)^{\ast}$. We now prove that $\Phi(\ell_{sch}^{q}(\G))=\left(\ell_{sch}^{p}(\G)\right)^{\ast}$. Note that it is enough to prove that $\Phi(\ell_{sch}^{q}(\G))$ is dense in $\left(\ell_{sch}^{p}(\G)\right)^{\ast}$. Let $H\in\left(\ell_{sch}^{p}(\G)\right)^{\ast\ast}$ be given such that $H(\phi)=0,$ for any $\phi\in\Phi(\ell_{sch}^{q}(\G)).$ Hence, for any $F\in\ell_{sch}^{q}(\G)$, we have that $\Phi(F)=\psi\in\Phi(\ell_{sch}^{q}(\G))$ for some $\psi.$ Thus, 
$$H(\Phi(F))=0,\quad \forall F\in\ell_{sch}^{q}(\G).$$ Since $\ell_{sch}^{p}(\G)$ is reflexive (cf. Proposition \ref{reflexivity}), we have
$$H(\Phi(F))=\Phi(F)(H)=\sum_{[\xi]\in\G}\dim(\xi)\mathrm{Tr}(F(\xi)H(\xi))=0,\quad \forall F\in\ell_{sch}^{q}(\G).$$
Choose $F\in\ell_{sch}^{q}(\G)$ such that 
$$F(\xi)=|H(\xi)|^{p-2}H^{\ast}(\xi),\quad \forall[\xi]\in\G.$$ Then, we have 
\begin{equation*}\begin{split}
&\lv F\rv_{\ell_{sch}^{q}(\G)}^{q}=\sum_{[\xi]\in\G}\dim(\xi)\lv F(\xi)\rv_{\mathcal{S}^{q}(\mathcal{H}_{\xi})}^{q}=\sum_{[\xi]\in\G}\dim(\xi)\mathrm{Tr}\left(|F(\xi)|^{q}\right)\\
&=\sum_{[\xi]\in\G}\dim(\xi)\mathrm{Tr}\left(\left||H(\xi)|^{p-2}H^{\ast}(\xi)\right|^{q}\right)=\sum_{[\xi]\in\G}\dim(\xi)\mathrm{Tr}\left(\left(\left(|H(\xi)|^{p-2}H^{\ast}(\xi)\right)^{\ast}|H(\xi)|^{p-2}H^{\ast}(\xi)\right)^{q/2}\right)\\
&=\sum_{[\xi]\in\G}\dim(\xi)\mathrm{Tr}\left(\left(H(\xi)|H(\xi)|^{2(p-2)}H^{\ast}(\xi)\right)^{q/2}\right)=\sum_{[\xi]\in\G}\dim(\xi)\mathrm{Tr}\left(\left(H^{\ast}(\xi)H(\xi)|H(\xi)|^{2(p-2)}\right)^{q/2}\right)\\
&=\sum_{[\xi]\in\G}\dim(\xi)\mathrm{Tr}\left(\left(|H(\xi)|^{2(p-2)+2}\right)^{q/2}\right)=\sum_{[\xi]\in\G}\dim(\xi)\mathrm{Tr}\left(|H(\xi)|^{p}\right)\\
&=\sum_{[\xi]\in\G}\dim(\xi)\lv H(\xi)\rv_{\mathcal{S}^{p}(\mathcal{H}_{\xi})}^{p}=\lv H\rv_{\ell_{sch}^{p}(\G)}^{p}<\infty.
\end{split}\end{equation*}
Here to obtain the sixth equality, we used the following fact that we name as \enquote{Fact A} for future reference. Note that $\mathrm{Tr}(AB)=\mathrm{Tr}(BA)$ for finite rank operators $A, B$ \cite[Section I.8.2]{gohberg1969translations}. Then, for a natural number $k$, one can show that $$\mathrm{Tr}\left((AB)^k\right)=\mathrm{Tr}\left(A(BA)^{k-1}B\right)=\mathrm{Tr}\left((BA)^{k-1}BA\right)=\mathrm{Tr}\left((BA)^{k}\right).$$ Hence, $\mathrm{Tr}(p(AB))=\mathrm{Tr}(p(BA))$ for any polynomial $p$ defined on the compact interval which contains the spectrum of both $A$ and $B$. Therefore, by Stone-Weierstrass theorem, one has $\mathrm{Tr}(f(AB))=\mathrm{Tr}(f(BA)),$ for any continuous real-valued function $f$ on the same compact interval. Especially, the same holds for $f(t)=t^{\frac{q}{2}}.$

Continuing the proof, we have 
\begin{equation*}\begin{split}
0&=\sum_{[\xi]\in\G}\dim(\xi)\mathrm{Tr}\left(F(\xi)H(\xi)\right)=\sum_{[\xi]\in\G}\dim(\xi)\mathrm{Tr}\left(|H(\xi)|^{p-2}H^{\ast}(\xi)H(\xi)\right)\\
&=\sum_{[\xi]\in\G}\dim(\xi)\mathrm{Tr}\left(|H(\xi)|^{p}\right)=\lv H\rv_{\ell_{sch}^{p}(\G)}^{p}.
\end{split}\end{equation*}
Therefore, $H\equiv 0.$ That means $\Phi(\ell_{sch}^{q}(\G))$ is dense in $\left(\ell_{sch}^{p}(\G)\right)^{\ast}$, which ends the proof. 
\end{proof}

We also have the following duality result.
\begin{proposition}\label{duality of the direct sum}
Let $w>0$. Let $1<p,q,r,s<\infty$ be given such that $\frac{1}{p}+\frac{1}{q}=1$ and $\frac{1}{r}+\frac{1}{s}=1$. Then 
\begin{equation}\label{dualdirectsum 1}
\left(\ell_{sch}^{p}(\G)\oplus_{r,w}\ell_{sch}^{p}(\G)\right)^{\ast}\cong\ell_{sch}^{q}(\G)\oplus_{s, \frac{1}{w}}\ell_{sch}^{q}(\G).
\end{equation}
In particular, 
\begin{equation}\label{dualdirectsum 2}
\left(\ell_{sch}^{p}(\G)\oplus_{r}\ell_{sch}^{p}(\G)\right)^{\ast}\cong\ell_{sch}^{q}(\G)\oplus_{s}\ell_{sch}^{q}(\G).
\end{equation}
\end{proposition}
\begin{proof} We only prove \eqref{dualdirectsum 1}, since \eqref{dualdirectsum 2} follows from \eqref{dualdirectsum 1} by assuming $w=1.$ By Proposition \ref{duality of lp}, we know that $\left(\ell_{sch}^{p}(\G)\right)^{\ast}\cong\ell_{sch}^{q}(\G)$. Hence, we prove \eqref{dualdirectsum 1} in the following form
$$\left(\ell_{sch}^{p}(\G)\oplus_{r,w}\ell_{sch}^{p}(\G)\right)^{\ast}\cong\left(\ell_{sch}^{p}(\G)\right)^{\ast}\oplus_{s, \frac{1}{w}}\left(\ell_{sch}^{p}(\G)\right)^{\ast}.$$

Define a map $T$ from $\left(\ell_{sch}^{p}(\G)\right)^{\ast}\oplus_{s, \frac{1}{w}}\left(\ell_{sch}^{p}(\G)\right)^{\ast}$ into $\left(\ell_{sch}^{p}(\G)\oplus_{r,w}\ell_{sch}^{p}(\G)\right)^{\ast}$ of the form
$$T(\varphi_{1},\varphi_{2})(H_{1}, H_{2})=\varphi_{1}(H_{1})+w^{\frac{1}{r}-\frac{1}{s}}\varphi_{2}(H_{2})$$
for $\varphi_{1}, \varphi_{2}\in \left(\ell_{sch}^{p}(\G)\right)^{\ast}$ and $H_{1}, H_{2}\in \ell_{sch}^{p}(\G).$ 

We aim to show that $T$ is an isometric isomorphism, i.e. a norm-preserving surjective mapping. We first show that $T$ is a norm-preserving mapping. We have
\begin{equation}\begin{split}\label{boundedness of T}
\left|T(\varphi_{1},\varphi_{2})(H_{1}, H_{2})\right|&\leq\left|\varphi_{1}(H_{1})\right|+\left|w^{\frac{1}{r}-\frac{1}{s}}\varphi_{2}(H_{2})\right|\\
&\leq\lv\varphi_{1}\rv_{\left(\ell_{sch}^{p}(\G)\right)^{\ast}}\lv H_{1}\rv_{\ell_{sch}^{p}(\G)}+w^{-\frac{1}{s}}\lv\varphi_{2}\rv_{\left(\ell_{sch}^{p}(\G)\right)^{\ast}}w^{\frac{1}{r}}\lv H_{2}\rv_{\ell_{sch}^{p}(\G)}\\
& \leq \left(\lv\varphi_{1}\rv_{\left(\ell_{sch}^{p}(\G)\right)^{\ast}}^{s}+\left(w^{-\frac{1}{s}}\right)^{s}\lv\varphi_{2}\rv_{\left(\ell_{sch}^{p}(\G)\right)^{\ast}}^{s}\right)^{\frac{1}{s}}\cdot\left(\lv H_{1}\rv_{\ell_{sch}^{p}(\G)}^{r}+\left(w^{\frac{1}{r}}\right)^{r}\lv H_{2}\rv_{\ell_{sch}^{p}(\G)}^{r}\right)^{\frac{1}{r}}\\
&=\lv \left(\varphi_{1}, \varphi_{2}\right)\rv_{\left(\ell_{sch}^{p}(\G)\right)^{\ast}\oplus_{s, \frac{1}{w}}\left(\ell_{sch}^{p}(\G)\right)^{\ast}} \lv (H_{1}, H_{2})\rv_{\left(\ell_{sch}^{p}(\G)\oplus_{r,w}\ell_{sch}^{p}(\G)\right)^{\ast}},
\end{split}\end{equation}
where the last inequality follows from the classical H\"{o}lder inequality.
Hence, 
$$\lv T(\varphi_{1}, \varphi_{2})\rv\leq \lv \left(\varphi_{1}, \varphi_{2}\right)\rv_{\left(\ell_{sch}^{p}(\G)\right)^{\ast}\oplus_{s, \frac{1}{w}}\left(\ell_{sch}^{p}(\G)\right)^{\ast}}.$$
Fix $\varepsilon>0$ and $h_{1},h_{2}\in \ell_{sch}^{p}(\G)$ such that 
\begin{equation}
\begin{split}\label{choosing epsilon}
\lv h_{1}\psrv=&\lv h_{2}\psrv=1\\
\lv\varphi_{1}\rv_{\left(\ell_{sch}^{p}(\G)\right)^{\ast}}\leq \left|\varphi_{1}(h_{1})\right|+\varepsilon,&\quad \lv\varphi_{2}\rv_{\left(\ell_{sch}^{p}(\G)\right)^{\ast}}\leq \left|\varphi_{2}(h_{2})\right|+\varepsilon.
\end{split}\end{equation}   

Without loss of generality we may assume that $\varphi_{1}(h_{1})\geq 0$ and $\varphi_{2}(h_{2})\geq 0$, since one can choose $h_{1}$ and $h_{2}$ by multiplying an appropriate complex number with absolute value 1 (if $z$ is the given complex number, then one can multiply it by $\frac{z^{\ast}}{|z|}$ with absolute value equal to 1). Let 
$$f_{1}=\frac{\lv\varphi_{1}\rv_{\left(\ell_{sch}^{p}(\G)\right)^{\ast}}^{s-1}}{\left(\lv\varphi_{1}\rv_{\left(\ell_{sch}^{p}(\G)\right)^{\ast}}^{s}+\frac{1}{w}\lv\varphi_{2}\rv_{\left(\ell_{sch}^{p}(\G)\right)^{\ast}}^{s}\right)^{\frac{1}{r}}}\cdot h_{1},\quad f_{2}=\frac{w^{-\frac{2}{r}}\lv\varphi_{2}\rv_{\left(\ell_{sch}^{p}(\G)\right)^{\ast}}^{s-1}}{\left(\lv\varphi_{1}\rv_{\left(\ell_{sch}^{p}(\G)\right)^{\ast}}^{s}+\frac{1}{w}\lv\varphi_{2}\rv_{\left(\ell_{sch}^{p}(\G)\right)^{\ast}}^{s}\right)^{\frac{1}{r}}}\cdot h_{2}.$$
Obviously, 
$$\lv (f_{1}, f_{2})\rv_{\left(\ell_{sch}^{p}(\G)\oplus_{r,w}\ell_{sch}^{p}(\G)\right)^{\ast}}=1.$$
Moreover, 
\begin{equation*}\begin{split}
T(\varphi_{1}, \varphi_{2})&(f_{1}, f_{2})=\varphi_{1}(f_{1})+w^{\frac{1}{r}-\frac{1}{s}}\varphi_{2}(f_{2})=\\
&=\frac{\lv\varphi_{1}\rv_{\left(\ell_{sch}^{p}(\G)\right)^{\ast}}^{s-1}}{\left(\lv\varphi_{1}\rv_{\left(\ell_{sch}^{p}(\G)\right)^{\ast}}^{s}+\frac{1}{w}\lv\varphi_{2}\rv_{\left(\ell_{sch}^{p}(\G)\right)^{\ast}}^{s}\right)^{\frac{1}{r}}}\cdot\varphi_{1}(h_{1})+\frac{w^{-\frac{2}{r}}w^{\frac{1}{r}-\frac{1}{s}}\lv\varphi_{2}\rv_{\left(\ell_{sch}^{p}(\G)\right)^{\ast}}^{s-1}}{\left(\lv\varphi_{1}\rv_{\left(\ell_{sch}^{p}(\G)\right)^{\ast}}^{s}+\frac{1}{w}\lv\varphi_{2}\rv_{\left(\ell_{sch}^{p}(\G)\right)^{\ast}}^{s}\right)^{\frac{1}{r}}}\cdot \varphi_{2}(h_{2})
\end{split}\end{equation*}
\begin{equation*}\begin{split}
&\geq \frac{\lv\varphi_{1}\rv_{\left(\ell_{sch}^{p}(\G)\right)^{\ast}}^{s-1}}{\left(\lv\varphi_{1}\rv_{\left(\ell_{sch}^{p}(\G)\right)^{\ast}}^{s}+\frac{1}{w}\lv\varphi_{2}\rv_{\left(\ell_{sch}^{p}(\G)\right)^{\ast}}^{s}\right)^{\frac{1}{r}}}\cdot\left(\lv\varphi_{1}\rv_{\left(\ell_{sch}^{p}(\G)\right)^{\ast}}-\varepsilon\right)\\
&\quad\quad\quad+\frac{\frac{1}{w}\lv\varphi_{2}\rv_{\left(\ell_{sch}^{p}(\G)\right)^{\ast}}^{s-1}}{\left(\lv\varphi_{1}\rv_{\left(\ell_{sch}^{p}(\G)\right)^{\ast}}^{s}+\frac{1}{w}\lv\varphi_{2}\rv_{\left(\ell_{sch}^{p}(\G)\right)^{\ast}}^{s}\right)^{\frac{1}{r}}}\cdot \left(\lv\varphi_{2}\rv_{\left(\ell_{sch}^{p}(\G)\right)^{\ast}}-\varepsilon\right)\\
&=\frac{\lv\varphi_{1}\rv_{\left(\ell_{sch}^{p}(\G)\right)^{\ast}}^{s}+\frac{1}{w}\lv\varphi_{2}\rv_{\left(\ell_{sch}^{p}(\G)\right)^{\ast}}^{s}}{\left(\lv\varphi_{1}\rv_{\left(\ell_{sch}^{p}(\G)\right)^{\ast}}^{s}+\frac{1}{w}\lv\varphi_{2}\rv_{\left(\ell_{sch}^{p}(\G)\right)^{\ast}}^{s}\right)^{\frac{1}{r}}}-\varepsilon\cdot \frac{\lv\varphi_{1}\rv_{\left(\ell_{sch}^{p}(\G)\right)^{\ast}}^{s-1}+\frac{1}{w}\lv\varphi_{2}\rv_{\left(\ell_{sch}^{p}(\G)\right)^{\ast}}^{s-1}}{\left(\lv\varphi_{1}\rv_{\left(\ell_{sch}^{p}(\G)\right)^{\ast}}^{s}+\frac{1}{w}\lv\varphi_{2}\rv_{\left(\ell_{sch}^{p}(\G)\right)^{\ast}}^{s}\right)^{\frac{1}{r}}},
\end{split}\end{equation*} 
where 
\begin{equation*}\begin{split}
\frac{\lv\varphi_{1}\rv_{\left(\ell_{sch}^{p}(\G)\right)^{\ast}}^{s}+\frac{1}{w}\lv\varphi_{2}\rv_{\left(\ell_{sch}^{p}(\G)\right)^{\ast}}^{s}}{\left(\lv\varphi_{1}\rv_{\left(\ell_{sch}^{p}(\G)\right)^{\ast}}^{s}+\frac{1}{w}\lv\varphi_{2}\rv_{\left(\ell_{sch}^{p}(\G)\right)^{\ast}}^{s}\right)^{\frac{1}{r}}}
&=\left(\lv\varphi_{1}\rv_{\left(\ell_{sch}^{p}(\G)\right)^{\ast}}^{s}+\frac{1}{w}\lv\varphi_{2}\rv_{\left(\ell_{sch}^{p}(\G)\right)^{\ast}}^{s}\right)^{\frac{1}{s}}\\
&=\lv (\varphi_{1}, \varphi_{2})\rv_{\left(\ell_{sch}^{p}(\G)\right)^{\ast}\oplus_{s, \frac{1}{w}}\left(\ell_{sch}^{p}(\G)\right)^{\ast}}.
\end{split}\end{equation*}
Hence, 
\begin{equation*}
T(\varphi_{1}, \varphi_{2})(f_{1}, f_{2})\geq \lv (\varphi_{1}, \varphi_{2})\rv_{\left(\ell_{sch}^{p}(\G)\right)^{\ast}\oplus_{s, \frac{1}{w}}\left(\ell_{sch}^{p}(\G)\right)^{\ast}}-\varepsilon\cdot \frac{\lv\varphi_{1}\rv_{\left(\ell_{sch}^{p}(\G)\right)^{\ast}}^{s-1}+\frac{1}{w}\lv\varphi_{2}\rv_{\left(\ell_{sch}^{p}(\G)\right)^{\ast}}^{s-1}}{\left(\lv\varphi_{1}\rv_{\left(\ell_{sch}^{p}(\G)\right)^{\ast}}^{s}+\frac{1}{w}\lv\varphi_{2}\rv_{\left(\ell_{sch}^{p}(\G)\right)^{\ast}}^{s}\right)^{\frac{1}{r}}}.
\end{equation*}
Since $\varepsilon>0$, in \eqref{choosing epsilon}, can be chosen arbitrarily small, one has  
$$T(\varphi_{1}, \varphi_{2})(f_{1}, f_{2})\geq \lv (\varphi_{1}, \varphi_{2})\rv_{\left(\ell_{sch}^{p}(\G)\right)^{\ast}\oplus_{s, \frac{1}{w}}\left(\ell_{sch}^{p}(\G)\right)^{\ast}}.$$
The last inequality together with \eqref{boundedness of T} implies that 
$$\lv T(\varphi_{1}, \varphi_{2})\rv= \lv \left(\varphi_{1}, \varphi_{2}\right)\rv_{\left(\ell_{sch}^{p}(\G)\right)^{\ast}\oplus_{s, \frac{1}{w}}\left(\ell_{sch}^{p}(\G)\right)^{\ast}},\quad \forall (\varphi_{1}, \varphi_{2})\in\left(\ell_{sch}^{p}(\G)\right)^{\ast}\oplus_{s, \frac{1}{w}}\left(\ell_{sch}^{p}(\G)\right)^{\ast},$$
which shows that $T$ is a norm-preserving mapping.

We now show that $T$ is a surjective mapping. Let $\Phi\in \left(\ell_{sch}^{p}(\G)\oplus_{r,w}\ell_{sch}^{p}(\G)\right)^{\ast}$ be a bounded linear functional. Then, for $H_{1}, H_{2}\in\ell_{sch}^{p}(\G)$, one has
$$\Phi(H_{1}, H_{2})=\Phi(H_{1},0)+\Phi(0, H_{2})=\Phi\arrowvert_{\ell_{sch}^{p}(\G)\oplus_{r,w}\{0\}}(H_{1}, 0)+w^{\frac{1}{r}-\frac{1}{s}}\Phi\arrowvert_{\{0\}\oplus_{r,w}\ell_{sch}^{p}(\G)}(0, w^{\frac{1}{s}-\frac{1}{r}}H_{2}).$$
Denote by
$$\varphi_{1}(H_{1})=\Phi\arrowvert_{\ell_{sch}^{p}(\G)\oplus_{r,w}\{0\}}(H_{1}, 0),\quad \varphi_{2}(H_{2})=\Phi\arrowvert_{\{0\}\oplus_{r,w}\ell_{sch}^{p}(\G)}(0, w^{\frac{1}{s}-\frac{1}{r}}H_{2}).$$
Hence, $\varphi_{1}\in\left(
\ell_{sch}^{p}(\G)\right)^{\ast}$, $\varphi_{2}\in\left(
\ell_{sch}^{p}(\G)\right)^{\ast}$ and $(\varphi_{1},\varphi_{2})\in \left(\ell_{sch}^{p}(\G)\right)^{\ast}\oplus_{s, \frac{1}{w}}\left(\ell_{sch}^{p}(\G)\right)^{\ast}.$ The linearity of $\varphi_{1}$ and $\varphi_{2}$ follows easily from the linearity of $\Phi$. The boundedness can be shown as follows
\begin{equation*}\begin{split}
\left|\varphi_{1}(H_{1})\right|&=\left|\Phi\arrowvert_{\ell_{sch}^{p}(\G)\oplus_{r,w}\{0\}}(H_{1}, 0)\right|\leq \lv \Phi\arrowvert_{\ell_{sch}^{p}(\G)\oplus_{r,w}\{0\}}\rv_{\left(\ell_{sch}^{p}(\G)\oplus_{r,w}\{0\}\right)^{\ast}}\lv(H_{1}, 0)\rv_{\ell_{sch}^{p}(\G)\oplus_{r,w}\{0\}}\\
&\leq \lv \Phi\rv_{\left(\ell_{sch}^{p}(\G)\oplus_{r,w}\ell_{sch}^{p}(\G)\right)^{\ast}}\lv H_{1}\rv_{\ell_{sch}^{p}(\G)},\quad \forall H_{1}\in \ell_{sch}^{p}(\G),
\end{split}\end{equation*}
and 
\begin{equation*}\begin{split}
\left|\varphi_{2}(H_{2})\right|&=\left|\Phi\arrowvert_{\{0\}\oplus_{r,w}\ell_{sch}^{p}(\G)}(0, w^{\frac{1}{s}-\frac{1}{r}}H_{2})\right|\leq \lv \Phi\arrowvert_{\{0\}\oplus_{r,w}\ell_{sch}^{p}(\G)}\rv_{\left(\{0\}\oplus_{r,w}\ell_{sch}^{p}(\G)\right)^{\ast}}\lv(0, w^{\frac{1}{s}-\frac{1}{r}}H_{2})\rv_{\{0\}\oplus_{r,w}\ell_{sch}^{p}(\G)}\\
&\leq w^{\frac{1}{s}}\lv \Phi\rv_{\left(\ell_{sch}^{p}(\G)\oplus_{r,w}\ell_{sch}^{p}(\G)\right)^{\ast}}\lv H_{2}\rv_{\ell_{sch}^{p}(\G)},\quad \forall H_{2}\in \ell_{sch}^{p}(\G).
\end{split}\end{equation*}
Lastly, 
\begin{equation*}\begin{split}
&\lv (\varphi_{1}, \varphi_{2}) \rv_{\left(\ell_{sch}^{p}(\G)\right)^{\ast}\oplus_{s, \frac{1}{w}}\left(\ell_{sch}^{p}(\G)\right)^{\ast}}=\left(\lv \varphi_{1}\rv_{\left(\ell_{sch}^{p}(\G)\right)^{\ast}}^{s}+\frac{1}{w}\lv \varphi_{2}\rv_{\left(\ell_{sch}^{p}(\G)\right)^{\ast}}^{s}\right)^{\frac{1}{s}}\\
&\leq \left(\lv \Phi\rv_{\left(\ell_{sch}^{p}(\G)\oplus_{r,w}\ell_{sch}^{p}(\G)\right)^{\ast}}^{s}+\frac{1}{w}\left(w^{\frac{1}{s}}\right)^{s}\lv \Phi\rv_{\left(\ell_{sch}^{p}(\G)\oplus_{r,w}\ell_{sch}^{p}(\G)\right)^{\ast}}^{s}\right)^{\frac{1}{s}}=2^{\frac{1}{s}}\lv \Phi\rv_{\left(\ell_{sch}^{p}(\G)\oplus_{r,w}\ell_{sch}^{p}(\G)\right)^{\ast}}<\infty,
\end{split}\end{equation*}
which completes the proof.
\end{proof}

\subsection{Complex interpolation of $\ell_{sch}^{p}(\G)$ spaces.}\label{subsectioninterpolation}

In this subsection present a complex interpolation result for spaces $\ell_{sch}^{p}(\G)$, $1\leq p<\infty$. For more detailed discussion on basics of complex interpolation, we refer the reader to \cite[Chapter 4]{lofstrom1976interpolation} 

\begin{proposition}\label{Schatten interpolation}
Let $1\leq p_{0},p_{1}<\infty$. Then,
$$\left(\ell_{sch}^{p_{0}}(\G), \ell_{sch}^{p_{1}}(\G)\right)_{\theta}=\ell_{sch}^{p}(\G),\quad 0<\theta<1,$$
with equal norms, where 
$$\frac{1}{p}=\frac{1-\theta}{p_{0}}+\frac{\theta}{p_{1}}.$$
\end{proposition}
\begin{proof} Firstly, let us recall that for a given finite-rank operator $K$ on a Hilbert space $H$, we understand $K^{z}, \mathrm{Re}z\geq 0$ as follows 
$$K^{z}=\sum_{j=1}^{N}\lambda_{j}^{z}\langle\cdot, \phi_{j}\rangle\phi_{j},\quad \lambda^{z}=e^{z \ln(\lambda)},$$
where $K=\sum_{j=1}^{N}\lambda_{j}\langle\cdot, \phi_{j}\rangle\phi_{j}$ is the spectral representation of $K.$

Note that it is sufficient to check that $\left(\ell_{sch}^{p_{0}}(\G), \ell_{sch}^{p_{1}}(\G)\right)_{\theta}$ can be identified with $\ell_{sch}^{p}(\G)$. First suppose that $H\in\ell_{sch}^{p}(\G)$ and without loss of generality assume that $\lv H\psrv=1.$ Define a function $f\in\mathcal{F}:=\mathcal{F}(\ell_{sch}^{p_{0}}(\G), \ell_{sch}^{p_{1}}(\G))$ from the strip $\{z:0\leq \mathrm{Re}z\leq 1\}$ into $\ell_{sch}^{p_{0}}(\G)+\ell_{sch}^{p_{1}}(\G)$ by the setting
$$f(z)(\xi)=\left|H^{\ast}(\xi)\right|^{\frac{p}{p(z)}-1}H(\xi),\quad \forall[\xi]\in\G,$$
where $$\frac{p}{p(z)}=\frac{p(1-z)}{p_{0}}+\frac{pz}{p_{1}},$$ with convention that $\frac{0}{0}:=0.$ It is obvious that $p(\theta)=p$ and $f(\theta)(\xi)=H(\xi),$ $\forall [\xi]\in\G,$ i.e. $f(\theta)=H.$ Moreover, 
$$\lv f\rv_{\mathcal{F}}:=\max\left\{\sup_{t\in\mathbb{R}} \lv f(it)\rv_{\ell_{sch}^{p_{0}}(\G)}, \sup_{t\in\mathbb{R}} \lv f(1+it)\rv_{\ell_{sch}^{p_{1}}(\G)}\right\}=1.$$
Indeed,
\begin{equation*}\begin{split}
&\sup_{t\in\mathbb{R}}\lv f(it)\rv_{\ell_{sch}^{p_{0}}(\G)}=\sup_{t\in\mathbb{R}}\left(\sum_{[\xi]\in\G}\dim(\xi)\lv |H^{\ast}(\xi)|^{\frac{p}{p(it)}-1}H(\xi)\rv_{\mathcal{S}^{p_{0}}}^{p_{0}}\right)^{\frac{1}{p_{0}}}\\
&=\sup_{t\in\mathbb{R}}\left(\sum_{[\xi]\in\G}\dim(\xi)\mathrm{Tr} \left| |H^{\ast}(\xi)|^{\frac{p}{p(it)}-1}H(\xi)\right|^{p_{0}}\right)^{\frac{1}{p_{0}}}\\
&=\sup_{t\in\mathbb{R}}\left(\sum_{[\xi]\in\G}\dim(\xi)\mathrm{Tr} \left(H^{\ast}(\xi) \left(|H^{\ast}(\xi)|^{\frac{p}{p(it)}-1}\right)^{\ast} |H^{\ast}(\xi)|^{\frac{p}{p(it)}-1}H(\xi)\right)^{\frac{p_{0}}{2}}\right)^{\frac{1}{p_{0}}}\\
&=\sup_{t\in\mathbb{R}}\left(\sum_{[\xi]\in\G}\dim(\xi)\mathrm{Tr} \left(H^{\ast}(\xi) \left| |H^{\ast}(\xi)|^{\frac{p}{p(it)}-1}\right|^2 H(\xi)\right)^{\frac{p_{0}}{2}}\right)^{\frac{1}{p_{0}}}\\
&=\sup_{t\in\mathbb{R}}\left(\sum_{[\xi]\in\G}\dim(\xi)\mathrm{Tr} \left(H^{\ast}(\xi) |H^{\ast}(\xi)|^{2\mathrm{Re}\left(\frac{p}{p(it)}-1\right)} H(\xi)\right)^{\frac{p_{0}}{2}}\right)^{\frac{1}{p_{0}}}\\
&=\left(\sum_{[\xi]\in\G}\dim(\xi)\mathrm{Tr} \left(H^{\ast}(\xi) |H^{\ast}(\xi)|^{2(\frac{p}{p_{0}}-1)} H(\xi)\right)^{\frac{p_{0}}{2}}\right)^{\frac{1}{p_{0}}}=\left(\sum_{[\xi]\in\G}\dim(\xi)\mathrm{Tr} \left(|H^{\ast}(\xi)|^{2(\frac{p}{p_{0}}-1)} H(\xi) H^{\ast}(\xi)\right)^{\frac{p_{0}}{2}}\right)^{\frac{1}{p_{0}}}\\
&=\left(\sum_{[\xi]\in\G}\dim(\xi)\mathrm{Tr} \left(|H^{\ast}(\xi)|^{2(\frac{p}{p_{0}}-1)} |H^{\ast}(\xi)|^2|\right)^{\frac{p_{0}}{2}}\right)^{\frac{1}{p_{0}}}=\left(\sum_{[\xi]\in\G}\dim(\xi)\mathrm{Tr} \left(|H^{\ast}(\xi)|^{p}\right)\right)^{\frac{1}{p_{0}}}\\
&=\left(\sum_{[\xi]\in\G}\dim(\xi)\lv H^{\ast}(\xi)\rv_{\mathcal{S}^{p}(\mathcal{H}_{\xi})}^{p}\right)^{\frac{1}{p_{0}}}=\left(\sum_{[\xi]\in\G}\dim(\xi)\lv H(\xi)\rv_{\mathcal{S}^{p}(\mathcal{H}_{\xi})}^{p}\right)^{\frac{1}{p_{0}}}=\lv H\rv_{\ell_{sch}^{p}(\G)}^{\frac{p}{p_{0}}}=1.
\end{split}\end{equation*}
Here to obtain the seventh equality, we used \enquote{Fact A} in the proof of Proposition \ref{duality of lp} \ref{duality statement 2} that especially holds for $f(t)=t^{\frac{p_{0}}{2}}.$

Similarly, one can show that
$$\sup_{t\in\mathbb{R}}\lv f(1+it)\rv_{\ell_{sch}^{p_{1}}(\G)}=1.$$ Therefore, for a given $H\in\ell_{sch}^{p}(\G)$ there exists $f\in \mathcal{F}$ such that $f(\theta)=H$ and $\lv f\rv_{\mathcal{F}}=1$, which implies by the definition that $H\in\left(\ell_{sch}^{p_{0}}(\G), \ell_{sch}^{p_{1}}(\G)\right)_{\theta}$ and $\lv H\rv_{\left(\ell_{sch}^{p_{0}}(\G), \ell_{sch}^{p_{1}}(\G)\right)_{\theta}}\leq 1=\lv H\psrv.$

Now suppose that $H\in\left(\ell_{sch}^{p_{0}}(\G), \ell_{sch}^{p_{1}}(\G)\right)_{\theta}$ and without loss of generality assume that $\lv H\rv_{\left(\ell_{sch}^{p_{0}}(\G), \ell_{sch}^{p_{1}}(\G)\right)_{\theta}}=1.$  Let $\varepsilon>0$ be given. Then choose $f\in\mathcal{F}$ such that $\lv f\rv_{\mathcal{F}}<1+\varepsilon$ and $f(\theta)=H.$ In order to prove that $\lv H\psrv\leq 1$, by Proposition \ref{duality of lp} \ref{duality statement 1}, it is enough to show that $|\langle H, F\rangle_{\G}|\leq 1$ for any $F\in\ell_{sch}^{q}(\G)$ with $\lv F\rv_{\ell_{sch}^{q}(\G)}=1,$ where $q$ is the conjugate index of $p.$

Consider $F\in\ell_{sch}^{q}(\G)$ with $\lv F\rv_{\ell_{sch}^{q}(\G)}=1$ and define a function $g$ on a strip $\{z: 0\leq \mathrm{Re} z\leq 1\}$ into $\ell_{sch}^{q_{0}}(\G)+\ell_{sch}^{q_{1}}(\G)$ in the following form
$$g(z)(\xi)=\left|F^{\ast}(\xi)\right|^{\frac{q}{q(z)}-1}F(\xi),\quad \forall[\xi]\in\G,$$
where $q_{0}, q_{1}$ are conjugate indices of $p_{0}, p_{1}$, respectively, and 
$$\frac{q}{q(z)}=\frac{q(1-z)}{q_{0}}+\frac{qz}{q_{1}},$$ with convention that $\frac{0}{0}:=0.$
We obviously have that $g(\theta)=F$ and the fact that $\lv F\rv_{\ell_{sch}^{q}(\G)}=1$ together with similar calculations as in the first part of the proof imply that 
\begin{equation}\begin{split}
\lv g(it)\rv_{\ell_{sch}^{q_{0}}(\G)}=\lv g(1+it)\rv_{\ell_{sch}^{q_{1}}(\G)}=1,\quad t\in\mathbb{R}.   
\end{split}\end{equation} 
Hence, $\lv g\rv_{\mathcal{F}}=1.$

Let $$h(z)=\langle f(z), g(z)\rangle_{\G},$$ (see Proposition \ref{duality of lp} for the definition of the given bracket) which is a function from the strip $\{z: 0\leq \mathrm{Re}z\leq 1\}$ into $\mathbb{R}.$ Then, by H\"older's inequality in Schatten-von Neumann ideals, one has 
\begin{equation}\begin{split}\label{complex eq 1}
|h(it)|&=\left|\langle f(it), g(it)\rangle_{\G}\right|=\left|\sum_{[\xi]\in\G}\dim(\xi)\mathrm{Tr}\left(f(it)(\xi)g(it)(\xi)\right)\right|\leq \sum_{[\xi]\in\G}\dim(\xi)\mathrm{Tr}\left|f(it)(\xi)g(it)(\xi)\right|\\
&\leq \sum_{[\xi]\in\G}\dim(\xi)^{\frac{1}{p_{0}}}\lv f(it)(\xi)\rv_{\mathcal{S}^{p_{0}}(\mathcal{H}_{\xi})} \dim(\xi)^{\frac{1}{q_{0}}}\lv g(it)(\xi)\rv_{\mathcal{S}^{q_{0}}(\mathcal{H}_{\xi})}\\
&\leq \left(\sum_{[\xi]\in\G}\dim(\xi)\lv f(it)(\xi)\rv_{\mathcal{S}^{p_{0}}(\mathcal{H}_{\xi})}^{p_{0}}\right)^{\frac{1}{p_{0}}}\left(\sum_{[\xi]\in\G}\dim(\xi)\lv g(it)(\xi)\rv_{\mathcal{S}^{q_{0}}(\mathcal{H}_{\xi})}^{q_{0}}\right)^{\frac{1}{q_{0}}}\\
&\leq \lv f(it)\rv_{\ell_{sch}^{p_{0}}(\G)}\lv g(it)\rv_{\ell_{sch}^{q_{0}}(\G)}\leq \lv f\rv_{\mathcal{F}}\lv g\rv_{\mathcal{F}}<1+\varepsilon, \quad t\in\mathbb{R}.
\end{split}\end{equation}
Similarly, we have
\begin{equation}\label{complex eq 2}
    |h(1+it)|< 1+\varepsilon, \quad t\in\mathbb{R}.
\end{equation}
Therefore, \eqref{complex eq 1} and \eqref{complex eq 2} together with the principle of the maximum for Banach space valued functions (see, for example, \cite[Section III.14, p. 230]{dunford1988linear} or \cite[Section 3.13, p. 100]{hille1996functional}) imply that 
$$|h(\theta+it)|=\left|\langle f(\theta+it),g(\theta+it)\rangle_{\G}\right|< 1+\varepsilon,\quad t\in\mathbb{R}.$$
In particular,
$$h(\theta)=\left|\langle f(\theta),g(\theta)\rangle_{\G}\right|=\left|\langle H,F\rangle_{\G}\right|< 1+\varepsilon.$$
Since $F$ is arbitrary, it follows that $\lv H\rv_{\ell_{sch}^{p}(\G)}< 1+\varepsilon=\lv H\rv_{\left(\ell_{sch}^{p_{0}}(\G), \ell_{sch}^{p_{1}}(\G)\right)_{\theta}}+\varepsilon.$ Since $\varepsilon>0$ is arbitrary, we have $\lv H\rv_{\ell_{sch}^{p}(\G)}\leq \lv H\rv_{\left(\ell_{sch}^{p_{0}}(\G), \ell_{sch}^{p_{1}}(\G)\right)_{\theta}},$ completing the proof.
\end{proof}

\subsection{Clarkson's inequalities and Kadec-Klee property}
We now prove the following Clarkson inequalities in $\ell_{sch}^{p}(\widehat{G})$ using the simple interpolation method which was first developed by Boas \cite{boas1940some} and M. Klaus \cite[p. 15]{simon2005trace}. These inequalities can further be used to show that the spaces $\ell_{sch}^{p}(\G),$ $1<p<\infty,$ satisfy the Kadec-Klee property.

\begin{proposition}\label{new Clarkson type inequality}
Let $1<p,q<\infty$ with $\frac{1}{p}+\frac{1}{q}=1$. Then, for any $H_{1}, H_{2}\in \ell_{sch}^{p}(\widehat{G})$, one has the following inequalities:
\begin{enumerate}[label=(\roman*)]
    \item \label{new first inequality} If $1<p\leq 2$, then 
    \begin{equation*}\begin{split}
    \Biggl(\lv \frac{H_{1}+H_{2}}{2} \psrv^{q}+\lv \frac{H_{1}-H_{2}}{2} \psrv^{q}\Biggr)^{1/q}\leq \left(\frac{1}{2}\left(\lv H_{1}\psrv^{p}+\lv H_{2}\psrv^{p}\right)\right)^{1/p};
    \end{split}\end{equation*}
    \item \label{new second inequality} If $2\leq p< \infty$, then \begin{equation*}\begin{split}
    \Biggl(\lv \frac{H_{1}+H_{2}}{2} \psrv^{p}+\lv \frac{H_{1}-H_{2}}{2} \psrv^{p}\Biggr)^{1/p}\leq \left(\frac{1}{2}\left(\lv H_{1}\psrv^{q}+\lv H_{2}\psrv^{q}\right)\right)^{1/q}.
    \end{split}\end{equation*}
\end{enumerate}
\end{proposition}

\begin{proof} 
\ref{new first inequality}. We first define the map $T$ from $\ell_{sch}^{p}(\G)\oplus_{p}\ell_{sch}^{p}(\G)$ into $\ell_{sch}^{p}(\G)\oplus_{q}\ell_{sch}^{p}(\G)$ by setting

$$T(H_{1},H_{2})(\xi,\xi):=\left(\frac{H_{1}(\xi)+H_{2}(\xi)}{2},\frac{H_{1}(\xi)-H_{2}(\xi)}{2}\right).$$
Note that in order to prove \ref{new first inequality}, we have to show that 
$$\lv T\rv_{\ell_{sch}^{p}(\G)\oplus_{p}\ell_{sch}^{p}(\G)\rightarrow\ell_{sch}^{p}(\G)\oplus_{q}\ell_{sch}^{p}(\G)}\leq 2^{-1/p}.$$

By Proposition \ref{Schatten interpolation}, one has 
\begin{equation}\label{sch int}
\ell_{sch}^{p}(\G)=\left(\ell_{sch}^{1}(\G), \ell_{sch}^{2}(\G)\right)_{\theta}, \quad \frac{1}{p}=\frac{1-\theta}{1}+\frac{\theta}{2}.
\end{equation}
Hence, by Proposition \ref{calderon interpolation 1} and \eqref{sch int}, we have 
\begin{equation}\begin{split}\label{intereq1}
\ell_{sch}^{p}(\G)\oplus_{p}\ell_{sch}^{p}(\G)&=\left(\ell_{sch}^{1}(\G), \ell_{sch}^{2}(\G)\right)_{\theta}\oplus_{p}\left(\ell_{sch}^{1}(\G), \ell_{sch}^{2}(\G)\right)_{\theta}\\
&=\left(\ell_{sch}^{1}(\G)\oplus_{1}\ell_{sch}^{1}(\G), \ell_{sch}^{2}(\G)\oplus_{2}\ell_{sch}^{2}(\G)\right)_{\theta}. 
\end{split}\end{equation}
Note that 
$$\frac{1}{q}=1-\frac{1}{p}=1-\frac{1-\theta}{1}-\frac{\theta}{2}=\frac{1-\theta}{\infty}+\frac{\theta}{2}.$$
Hence, again by Proposition \ref{calderon interpolation 1} and \eqref{sch int}, we have 
\begin{equation}\begin{split}\label{intereq2}
\ell_{sch}^{p}(\G)\oplus_{q}\ell_{sch}^{p}(\G)&=\left(\ell_{sch}^{1}(\G), \ell_{sch}^{2}(\G)\right)_{\theta}\oplus_{q}\left(\ell_{sch}^{1}(\G), \ell_{sch}^{2}(\G)\right)_{\theta}\\
&=\left(\ell_{sch}^{1}(\G)\oplus_{\infty}\ell_{sch}^{1}(\G), \ell_{sch}^{2}(\G)\oplus_{2}\ell_{sch}^{2}(\G)\right)_{\theta}. 
\end{split}\end{equation}

Consider the operators $T_{1}$ from $\ell_{sch}^{1}(\G)\oplus_{1}\ell_{sch}^{1}(\G)$ into $\ell_{sch}^{1}(\G)\oplus_{\infty}\ell_{sch}^{1}(\G)$ and $T_{2}$ from $\ell_{sch}^{2}(\G)\oplus_{2}\ell_{sch}^{2}(\G)$ into $\ell_{sch}^{2}(\G)\oplus_{2}\ell_{sch}^{2}(\G)$, which are defined by the same action as an operator $T$ above in their respective domains. 

Then, one has
\begin{equation*}\begin{split}
\| T_{1}(H_{1},H_{2})&\|_{\ell_{sch}^{1}(\G)\oplus_{\infty}\ell_{sch}^{1}(\G)}=\max\left\{ \lv \frac{H_{1}(\xi)+H_{2}(\xi)}{2}\rv_{\ell_{sch}^{1}(\G)}, \lv \frac{H_{1}(\xi)-H_{2}(\xi)}{2}\rv_{\ell_{sch}^{1}(\G)}\right\}\\
&\leq \frac{1}{2}\left(\lv H_{1}(\xi)\rv_{\ell_{sch}^{1}(\G)}+ \lv H_{2}(\xi)\rv_{\ell_{sch}^{1}(\G)}\right)=2^{-1}\lv (H_{1},H_{2})\rv_{\ell_{sch}^{1}(\G)\oplus_{1}\ell_{sch}^{1}(\G)}.
\end{split}\end{equation*} Hence, 
\begin{equation}\label{t1}
\lv T_{1}\rv_{\ell_{sch}^{1}(\G)\oplus_{1}\ell_{sch}^{1}(\G)\rightarrow\ell_{sch}^{1}(\G)\oplus_{\infty}\ell_{sch}^{1}(\G)}\leq 2^{-1}.
\end{equation}

On the other hand, we have  
\begin{equation*}\begin{split}
\| T_2(H_{1},H_{2})&\|_{\ell_{sch}^{2}(\G)\oplus_{2}\ell_{sch}^{2}(\G)}=\left(\lv \frac{H_{1}(\xi)+H_{2}(\xi)}{2}\rv_{\ell_{sch}^{2}(\G)}^{2}+\lv \frac{H_{1}(\xi)-H_{2}(\xi)}{2}\rv_{\ell_{sch}^{2}(\G)}^{2}\right)^{1/2}\\
&=2^{-1/2}\left(\lv H_{1}(\xi)\rv_{\ell_{sch}^{2}(\G)}^{2}+\lv H_{2}(\xi)\rv_{\ell_{sch}^{2}(\G)}^{2}\right)^{1/2}=2^{-1/2}\lv (H_{1},H_{2})\rv_{\ell_{sch}^{2}(\G)\oplus_{2}\ell_{sch}^{2}(\G)}.
\end{split}\end{equation*} Hence, 
\begin{equation}\label{t2}
\lv T_{2}\rv_{\ell_{sch}^{2}(\G)\oplus_{2}\ell_{sch}^{2}(\G)\rightarrow\ell_{sch}^{2}(\G)\oplus_{2}\ell_{sch}^{2}(\G)}\leq 2^{-1/2}.
\end{equation}

By \eqref{intereq1} and \eqref{intereq2}, note again that $\ell_{sch}^{p}(\G)\oplus_{p}\ell_{sch}^{p}(\G)$ and $\ell_{sch}^{p}(\G)\oplus_{q}\ell_{sch}^{p}(\G)$ are pair of interpolation spaces of exponent $\theta$ for $\ell_{sch}^{1}(\G)\oplus_{1}\ell_{sch}^{1}(\G)$ and $\ell_{sch}^{2}(\G)\oplus_{2}\ell_{sch}^{2}(\G)$, respectively, $\ell_{sch}^{1}(\G)\oplus_{\infty}\ell_{sch}^{1}(\G)$ and $\ell_{sch}^{2}(\G)\oplus_{2}\ell_{sch}^{2}(\G)$. Therefore, by \cite[Corollary 5.5.4]{lofstrom1976interpolation} and the inequalities \eqref{t1}, \eqref{t2}, it follows that $T$ is bounded from $\ell_{sch}^{p}(\G)\oplus_{p}\ell_{sch}^{p}(\G)$ into $\ell_{sch}^{p}(\G)\oplus_{q} \ell_{sch}^{p}(\G)$ for every $1<p<2$ with norm
$$\lv T\rv_{\ell_{sch}^{p}(\G)\oplus_{p}\ell_{sch}^{p}(\G)\rightarrow\ell_{sch}^{p}(\G)\oplus_{q}\ell_{sch}^{p}(\G)}\leq \lv T_{1}\rv_{\ell_{sch}^{1}(\G)\oplus_{1}\ell_{sch}^{1}(\G)\rightarrow\ell_{sch}^{1}(\G)\oplus_{\infty}\ell_{sch}^{1}(\G)}^{1-\theta}\lv T_{2}\rv_{\ell_{sch}^{2}(\G)\oplus_{2}\ell_{sch}^{2}(\G)\rightarrow\ell_{sch}^{2}(\G)\oplus_{2}\ell_{sch}^{2}(\G)}^{\theta}$$
$$\leq \left(2^{-1}\right)^{1-\theta} \left(2^{-1/2}\right)^\theta=2^{-1+\theta-\theta/2}=2^{-1+\theta/2}=2^{-1/p},$$
which ends the proof of \ref{new first inequality}.

\ref{new second inequality}. Define the map $S$ from $\ell_{sch}^{p}(\G)\oplus_{q}\ell_{sch}^{p}(\G)$ into $\ell_{sch}^{p}(\G)\oplus_{p}\ell_{sch}^{p}(\G)$ by setting
$$S(H_{1},H_{2})(\xi,\xi):=\left(\frac{H_{1}(\xi)+H_{2}(\xi)}{2},\frac{H_{1}(\xi)-H_{2}(\xi)}{2}\right).$$ 
Note that it is enough to show that 
\begin{equation}\label{enough to prove 1}
\lv S\rv_{\ell_{sch}^{p}(\G)\oplus_{q}\ell_{sch}^{p}(\G)\rightarrow \ell_{sch}^{p}(\G)\oplus_{p}\ell_{sch}^{p}(\G)}\leq 2^{-1/q}.    
\end{equation}

Let $(H_{1},H_{2})\in\ell_{sch}^{p}(\G)\oplus_{q}\ell_{sch}^{p}(\G)$ and $(F_{1},F_{2})\in\ell_{sch}^{q}(\G)\oplus_{p}\ell_{sch}^{q}(\G)$, then (see, \cite[Definition 10.3.29]{ruzhansky2010pseudo} for more information on duality bracket)
\begin{align*}
\langle S(H_{1}, H_{2}), (F_{1},F_{2})\rangle_{\G}=&\left\langle \left(\frac{H_{1}+H_{2}}{2}, \frac{H_{1}-H_{2}}{2}\right), (F_{1},F_{2})\right\rangle_{\G}\\
:=\sum_{[\xi]\in\G}\Dim(\xi) &\mathrm{Tr}\left(\frac{H_{1}(\xi)+H_{2}(\xi)}{2}\cdot F_{1}(\xi)\right)+\sum_{[\xi]\in\G}\Dim(\xi) \mathrm{Tr}\left(\frac{H_{1}(\xi)-H_{2}(\xi)}{2}\cdot F_{2}(\xi)\right)\\
=\sum_{[\xi]\in\G}\Dim(\xi) &\mathrm{Tr}\left(H_{1}(\xi)\cdot\frac{F_{1}(\xi)+F_{2}(\xi)}{2}\right)+\sum_{[\xi]\in\G}\Dim(\xi) \mathrm{Tr}\left(H_{2}(\xi)\cdot\frac{F_{1}(\xi)-F_{2}(\xi)}{2}\right)\\
=\Bigl\langle (H_{1},H_{2})&, \left(\frac{F_{1}+F_{2}}{2}, \frac{F_{1}-F_{2}}{2}, \right)\Bigr\rangle_{\G}.
\end{align*}
Therefore, the Banach space adjoint $S^{\ast}$ from $\left(\ell_{sch}^{p}(\G)\oplus_{p}\ell_{sch}^{p}(\G)\right)^{\ast}$ into $\left(\ell_{sch}^{p}(\G)\oplus_{q}\ell_{sch}^{p}(\G)\right)^{\ast}$ is also defined by the formula
$$S^{\ast}(F_{1},F_{2})(\xi,\xi):=\left(\frac{F_{1}(\xi)+F_{2}(\xi)}{2},\frac{F_{1}(\xi)-F_{2}(\xi)}{2}\right).$$ 

Moreover, by Proposition \ref{duality of the direct sum}, one has
\begin{equation}\begin{split}\label{duality identity}
&\left(\ell_{sch}^{p}(\G)\oplus_{p}\ell_{sch}^{p}(\G)\right)^{\ast}=\ell_{sch}^{q}(\G)\oplus_{q}\ell_{sch}^{q}(\G),\\
&\left(\ell_{sch}^{p}(\G)\oplus_{q}\ell_{sch}^{p}(\G)\right)^{\ast}=\ell_{sch}^{q}(\G)\oplus_{p}\ell_{sch}^{q}(\G).
\end{split}\end{equation}
Since $1<q\leq 2$, by \eqref{duality identity} and by \ref{new first inequality} (where $p$ and $q$ are interchanged), it follows that the operator $S^{\ast}$ acts from $\ell_{sch}^{q}(\G)\oplus_{q}\ell_{sch}^{q}(\G)$ into $\ell_{sch}^{q}(\G)\oplus_{p}\ell_{sch}^{q}(\G)$ with norm 
$$\lv S^{\ast}\rv_{\ell_{sch}^{q}(\G)\oplus_{q}\ell_{sch}^{q}(\G)\rightarrow\ell_{sch}^{q}(\G)\oplus_{p}\ell_{sch}^{q}(\G)}\leq 2^{-1/q}.$$
Hence, by the norm equality of an operator from Banach space to another and its adjoint on dual spaces \cite[Theorem 4.10]{rudin1973functional}, we finally have that 
\begin{equation*}\begin{split}
\lv S\rv&_{\ell_{sch}^{p}(\G)\oplus_{q}\ell_{sch}^{p}(\G)\rightarrow \ell_{sch}^{p}(\G)\oplus_{p}\ell_{sch}^{p}(\G)}=\lv S^{\ast}\rv_{\left(\ell_{sch}^{p}(\G)\oplus_{p}\ell_{sch}^{p}(\G)\right)^{\ast}\rightarrow\left(\ell_{sch}^{p}(\G)\oplus_{q}\ell_{sch}^{p}(\G)\right)^{\ast}}\\
&=\lv S^{\ast}\rv_{\ell_{sch}^{q}(\G)\oplus_{q}\ell_{sch}^{q}(\G)\rightarrow\ell_{sch}^{q}(\G)\oplus_{p}\ell_{sch}^{q}(\G)}\leq 2^{-1/q},
\end{split}\end{equation*}
which shows that \eqref{enough to prove 1} holds, and ends the proof.
\end{proof}

As a consequence of the Clarkson inequality given in Propositition \ref{new Clarkson type inequality}, one can prove that the spaces $\ell_{sch}^{p}(\G)$ have the following Kadec-Klee property for $1<p<\infty$ \cite[Addenda, p. 130]{simon2005trace}. 

\begin{corollary}
Let $1<p<\infty$ and $H,H_{n}\in\ell_{sch}^{p}(\G),$ $n\geq 1$. If $$\lv H_{n}\psrv\rightarrow\lv H\psrv,\quad n\rightarrow\infty,$$
and if $H_{n}\rightarrow H,$ in the weak topology $\sigma(\ell_{sch}^{p}(\G),\ell_{sch}^{q}(\G))$, where $q$ is the conjugate index of $p$, then
$$\lv H_{n}-H\psrv\rightarrow0,\quad n\rightarrow\infty.$$
\end{corollary}

\begin{proof}
First we prove that the Banach norm is weakly lower semicontinuous, i.e. if $(X, \|\cdot\|)$ is the Banach space and $X\ni x_{n}\rightarrow x\in X$ weakly, then $$\|x\|\leq \liminf_{x_{n}\rightarrow x}\|x_{n}\|.$$
By the Hahn-Banach theorem, there exists $\varphi\in X^{\ast}$ such that $\lv \varphi\rv=1$ and $\varphi(x)=\lv x\rv.$ Since $x_{n}$ converges to $x$ weakly, we have
$$\lv x\rv=\varphi(x)=\liminf_{x_{n}\rightarrow x}\varphi(x_{n})=\liminf_{x_{n}\rightarrow x}|\varphi(x_{n})|\leq \liminf_{x_{n}\rightarrow x} \lv \varphi\rv \lv x_{n}\rv=\liminf_{x_{n}\rightarrow x} \lv x_{n}\rv.$$

By the assumption that $H_{n}$ weakly converges to $H$ and by the above fact that any Banach norm is weakly lower semicontinuous, it follows that 
$$2\lv H\psrv\leq \liminf_{n\rightarrow\infty}\lv H_{n}+H\psrv\leq \limsup_{n\rightarrow\infty}\lv H_{n}+H\psrv$$
$$\leq \limsup_{n\rightarrow\infty}\left(\lv H_{n}\psrv+\lv H\psrv\right)=2\lv H\psrv.$$
In particular 
$$\lv H_{n}+H\psrv\rightarrow2\lv H\psrv,\quad n\rightarrow\infty.$$

Assume the contrary that $\lv H_{n}-H\psrv$ does not converge to zero as $n$ goes to infinity. Then, passing to the subsequence if necessary, there exists $\varepsilon>0$ such that 
$$\lv H_{n}-H\psrv\geq \varepsilon\quad\text{for all} \quad n\geq 1.$$
Therefore, if $1<p\leq 2$, by Proposition \ref{new Clarkson type inequality} \ref{new first inequality}, one has 
\begin{multline*}
\frac{\varepsilon^{q}}{2^{q}}\leq \lv \frac{H_{n}-H}{2}\psrv^{q}\leq \left(\frac{1}{2}\left(\lv H_{n}\psrv^{p}+\lv H\psrv^{p}\right)\right)^{q/p}-\lv \frac{H_{n}+H}{2}\psrv^{q}\rightarrow 0,\quad n\rightarrow\infty.
\end{multline*}
Furthermore, if $2\leq p<\infty$, by Proposition \ref{new Clarkson type inequality} \ref{new second inequality}, one has 
\begin{multline*}
\frac{\varepsilon^{p}}{2^{p}}\leq \lv \frac{H_{n}-H}{2}\psrv^{p}\leq \left(\frac{1}{2}\left(\lv H_{n}\psrv^{q}+\lv H\psrv^{q}\right)\right)^{p/q}-\lv \frac{H_{n}+H}{2}\psrv^{p}\rightarrow 0,\quad n\rightarrow\infty.
\end{multline*}
Note that both cases lead to the contradiction. Hence, 
$$\lv H_{n}-H\psrv\rightarrow0,\quad n\rightarrow\infty,$$
completing the proof.
\end{proof}

\subsection{Geometric properties of $\ell_{sch}^{p}(\G)$}\label{geometric properties}

In this subsection, we consider various geometric properties such as uniform smoothness, uniform convexity as well as type and cotype properties of the family of $\ell^{p}$-spaces $\ell_{sch}^{p}(\widehat{G})$ based on the Schatten-von Neumann ideals, which is defined in Definition \ref{Schatten family of lp spaces}.

As in the theory of classical $l^p$-spaces, the spaces $\ell_{sch}^{p}(\G)$ are uniformly convex and uniformly smooth, which is stated in the following 
\begin{theorem}\label{mainth0} The space $\ell_{sch}^{p}(\widehat{G})$ is uniformly convex and uniformly smooth for $1<p<\infty$. Moreover, $\ell_{sch}^{p}(\G)$ is reflexive for $1<p<\infty.$
\end{theorem}
To prove Theorem \ref{mainth0}, we first present the following lemma on lower (resp. upper) estimates for the modulus of convexity (resp. smoothness) similar to the case of classical $\ell^p$-spaces.
\begin{lemma}\label{lemma0}
Let $1<p<\infty$, $0<\varepsilon\leq2$ and $t>0.$ Let $q$ be the conjugate index of $p$, i.e., $\frac{1}{p}+\frac{1}{q}=1.$
\begin{enumerate}[label=(\roman*)]
    \item \label{lemma0 inequality 1} If $1<p\leq 2,$ then 
    $$\delta_{\ell_{sch}^{p}(\widehat{G})}(\varepsilon)\geq \frac{\varepsilon^{q}}{q\cdot 2^{q}},\quad \rho_{\ell_{sch}^{p}(\widehat{G})}(t)\leq \frac{t^{p}}{p}.$$
    \item \label{lemma0 inequality 2}If $2<p< \infty,$ then 
    $$\delta_{\ell_{sch}^{p}(\widehat{G})}(\varepsilon)\geq \frac{\varepsilon^{p}}{p\cdot 2^{p}},\quad \rho_{\ell_{sch}^{p}(\widehat{G})}(t)\leq \frac{t^{q}}{q}.$$
\end{enumerate}
\end{lemma}

\begin{proof}
(i). Let $1<p\leq 2$ be given. We first prove the inequality for the modulus of convexity. Assume that $H_{1},H_{2}\in \ell_{sch}^{p}(\widehat{G})$ are such that $\lv H_{1}\psrv=\lv H_{2}\psrv=1$ and $\lv H_{1}-H_{2}\psrv=\varepsilon.$ Therefore, by Proposition \ref{new Clarkson type inequality} \ref{new first inequality}, one has 
$$\left(\lv \frac{H_{1}+H_{2}}{2} \psrv^{q}+\left(\frac{\varepsilon}{2}\right)^{q}\right)^{1/q}=\left(\lv \frac{H_{1}+H_{2}}{2} \psrv^{q}+\lv \frac{H_{1}-H_{2}}{2} \psrv^{q}\right)^{1/q}$$
$$\leq \left(\frac{1}{2}\left(\lv H_{1}\psrv^{p}+\lv H_{2}\psrv^{p}\right)\right)^{1/p}=1.$$
Hence, 
$$\lv \frac{H_{1}+H_{2}}{2}\psrv^{q}\leq 1-\left(\frac{\varepsilon}{2}\right)^{q},$$
which means
$$\lv \frac{H_{1}+H_{2}}{2}\psrv\leq \left(1-\left(\frac{\varepsilon}{2}\right)^{q}\right)^{1/q}\leq 1-\frac{\varepsilon^{q}}{q 2^{q}},$$
where the last inequality follows from Bernoulli's inequality, that is $(1+x)^{t}\leq 1+tx$ for every real number $0\leq t\leq 1$ and $x\geq -1.$ Finally, writing the last expression as 
$$\frac{\varepsilon^{q}}{q 2^{q}}\leq 1-\lv \frac{H_{1}+H_{2}}{2}\psrv$$
and taking the infimum over all $H_{1}, H_{2}\in \ell_{sch}^{p}(\widehat{G})$ satisfying the initial assumptions, we have $$\delta_{\ell_{sch}^{p}(\widehat{G})}(\varepsilon)\geq \frac{\varepsilon^{q}}{q 2^{q}},\quad 0<\varepsilon\leq2.$$

Now, we prove the inequality for the modulus of smoothness. Assume that $H_{1}, H_{2}\in \ell_{sch}^{p}(\widehat{G})$ are such that $\lv H_{1}\psrv=\lv H_{2}\psrv=1.$ Note that a function $f(x)=x^q,$ $x>0,$ is convex, since $q\geq 2$. Hence, assuming that 
$$x_{1}=\lv H_{1}+tH_{2}\psrv,\quad x_{2}=\lv H_{1}-tH_{2}\psrv,$$
by the definition, one has 
\begin{equation*}\begin{split}
f\left(\frac{x_{1}+x_{2}}{2}\right)&=\left(\frac{\lv H_{1}+tH_{2}\psrv+\lv H_{1}-tH_{2}\psrv}{2}\right)^{q}\\
&\leq \frac{f\left(x_{1}\right)}{2}+\frac{f\left(x_{2}\right)}{2}= \frac{\lv H_{1}+tH_{2}\psrv^{q}}{2}+\frac{\lv H_{1}-tH_{2}\psrv^{q}}{2}\\
&=\frac{2^q}{2}\left(\lv \frac{H_{1}+tH_{2}}{2}\psrv^q+\lv \frac{H_{1}-tH_{2}}{2}\psrv^q\right).
\end{split}\end{equation*}
This can be rewritten as
$$\frac{\lv H_{1}+tH_{2}\psrv+\lv H_{1}-tH_{2}\psrv}{2}\leq \frac{2}{2^{1/q}}\left(\lv \frac{ H_{1}+tH_{2}}{2}\psrv^{q}+\lv \frac{H_{1}-tH_{2}}{2}\psrv^{q}\right)^{1/q}.$$
Further, applying Proposition \ref{new Clarkson type inequality} \ref{new first inequality} to the right hand side, we have
$$\frac{\lv H_{1}+tH_{2}\psrv+\lv H_{1}-tH_{2}\psrv}{2}\leq \frac{2}{2^{1/q}}\left(\frac{1}{2}\left(\lv H_{1}\psrv^{p}+\lv tH_{2}\psrv^{p}\right)\right)^{1/p}=(1+t^{p})^{1/p}.$$
Hence, one has
$$\frac{\lv H_{1}+tH_{2}\psrv+\lv H_{1}-tH_{2}\psrv}{2}-1\leq (1+t^{p})^{1/p}-1\leq \frac{t^{p}}{p},$$
where the last inequality follows again from Bernoulli's inequality. Finally, taking the supremum over all $H_{1},H_{2}\in \ell_{sch}^{p}(\widehat{G})$ satisfying the initial assumptions, we have 
$$\rho_{\ell_{sch}^{p}(\widehat{G})}\leq \frac{t^{p}}{p},\quad t>0.$$

(ii). The proof of the second part follows the same line of reasoning as in part (i) and the inequality in Proposition \ref{new Clarkson type inequality} \ref{new second inequality}, hence, we omit it.
\end{proof}

\begin{proof}[Proof of Theorem \ref{mainth0}] Let $1<p\leq 2$. Then, by Lemma \ref{lemma0} \ref{lemma0 inequality 1}, one has 
$$\delta_{\ell_{sch}^{p}(\widehat{G})}(\varepsilon)\leq \frac{\varepsilon^{q}}{q\cdot 2^{q}},\quad \forall \varepsilon>0$$
and
$$\lim_{t\rightarrow0}\frac{\rho_{\ell_{sch}^{p}(\widehat{G})}}{t}\leq \lim_{t\rightarrow0}\frac{t^{p-1}}{p}=0.$$
Hence, $\ell_{sch}^{p}(\widehat{G})$ is uniformly convex and uniformly smooth for $1<p
\leq 2$. The case when $2<p<\infty$ analogously follows from Lemma \ref{lemma1} \ref{lemma1 inequality 2}, hence, we omit it. 
The reflexivity follows from Proposition \ref{MilmanPettis}.
\end{proof}

We have the following modification of Prosposition \ref{new Clarkson type inequality} similar to \cite[Theorem 5.3]{pisier2003non}, which further helps us to provide optimal rate estimates on the modulus of convexity and the modulus of smoothness when $1<p<2$ and $2<p<\infty$, respectively.

\begin{theorem}\label{mainth2}
\begin{enumerate}[label=(\roman*)]
    \item\label{ineq1} If $2\leq p<\infty$, then for any $H_{1},H_{2}\in \ell_{sch}^{p}(\widehat{G})$, one has 
    \begin{equation*}\begin{split}
    \biggl(\frac{1}{2}\biggl(\lv H_{1}+H_{2}\psrv^{p}+\lv H_{1}-H_{2}\psrv^{p}\biggr)\biggr)^{1/p}\leq \left(\lv H_{1}\psrv^2+C_{p}\lv H_{2}\psrv^2\right)^{1/2},
    \end{split}\end{equation*}
    where $C_{p}$ is a positive constant depending only on $p$ and $C_{p}
    \leq 2p-1$.
    \item\label{ineq2} If $1<p\leq 2$, then for any $H_{1},H_{2}\in \ell_{sch}^{p}(\widehat{G})$, one has
    \begin{equation*}\begin{split}
    \Bigl(\lv H_{1}\psrv^2+c_{p}\lv H_{2}\psrv^2\Bigr)^{1/2}\leq \left(\frac{1}{2}\left(\lv H_{1}+H_{2}\psrv^{p}+\lv H_{1}-H_{2}\psrv^{p}\right)\right)^{1/p},
    \end{split}\end{equation*}
    where the constant $c_{p}$ depends only on $p$ and $c_{p}
    \geq \frac{p-1}{p+1}.$ 
\end{enumerate}
\end{theorem}

\begin{proof} \ref{ineq1}. Let $2\leq p<\infty.$ We first show that if \ref{ineq1} is valid for $p$, then it is also valid for $2p.$ Assume that \ref{ineq1} holds for $p$ and $H_{1},H_{2}\in \ell_{sch}^{2p}(\widehat{G}).$ Consider $A=H_{1}^{\ast}H_{1}+H_{2}^{\ast}H_{2}$ and $B=H_{1}^{\ast}H_{2}+H_{2}^{\ast}H_{1}$. Hence, it follows from the H\"older inequality (see Proposition \ref{Holder inequality}) that $A, B\in \ell_{sch}^{p}(\widehat{G}).$ 

Since $\lv x\rv_{\mathcal{S}^{p}(\mathcal{H})}=\left(\mathrm{Tr}(\lvert x\rvert^{p})\right)^{1/p},$ $x\in\mathcal{S}^{p}(\mathcal{H}),$ we have
\begin{equation*}\begin{split}
\lv H_{1}+H_{2}\rv_{\ell_{sch}^{2p}(\widehat{G})}^{2p}&=\sum_{[\xi]\in\G}\Dim(\xi)\lv H_{1}(\xi)+H_{2}(\xi)\rv_{\mathcal{S}^{2p}(\mathcal{H}_{\xi})}^{2p}=\sum_{[\xi]\in\G}\Dim(\xi) \mathrm{Tr}\left(\lvert H_{1}(\xi)+H_{2}(\xi)\rvert^{2p}\right)\\
&=\sum_{[\xi]\in\G}\Dim(\xi)\mathrm{Tr}\left(\left[(H_{1}(\xi)+H_{2}(\xi))^{\ast}(H_{1}(\xi)+H_{2}(\xi))\right]^{p}\right)\\
&=\sum_{[\xi]\in\G}\Dim(\xi)\mathrm{Tr}\left((A(\xi)+B(\xi))^{p}\right)=\sum_{[\xi]\in\G}\Dim(\xi)\mathrm{Tr}\left(\lvert A(\xi)+B(\xi)\rvert^{p}\right)\\
&=\sum_{[\xi]\in\G}\Dim(\xi)\lv A(\xi)+B(\xi)\rv_{\mathcal{S}^{p}(\mathcal{H}_{\xi})}^p=\lv A+B\rv_{\ell_{sch}^{p}(\G)}^{p}.
\end{split}\end{equation*}
Similarly, 
$$\lv H_{1}-H_{2}\rv_{\ell_{sch}^{2p}(\widehat{G})}^{2p}=\lv A-B\rv_{\ell_{sch}^{p}(\G)}^{p}.$$
Therefore, 
\begin{equation}\begin{split}\label{equality}
\frac{1}{2}\Bigl(\lv H_{1}+H_{2}\rv_{\ell_{sch}^{2p}(\widehat{G})}^{2p}+\lv H_{1}-H_{2}\rv_{\ell_{sch}^{2p}(\widehat{G})}^{2p}\Bigr)&=\frac{1}{2}\left(\lv A-B\rv_{\ell_{sch}^{p}(\G)}^{p}+\lv A-B\rv_{\ell_{sch}^{p}(\G)}^{p}\right)\\
&\leq \left(\lv A\rv_{\ell_{sch}^{p}(\G)}^2+C_{p}\lv B\rv_{\ell_{sch}^{p}(\G)}^2\right)^{\frac{p}{2}},
\end{split}\end{equation}
where the last inequality follows from \ref{ineq1} that was assumed to be true for $p.$ Note that
\begin{equation}\begin{split}\label{id1}
\lv A\rv_{\ell_{sch}^{p}(\G)}^2&=\lv H_{1}^{\ast}H_{1}+H_{2}^{\ast}H_{2}\rv_{\ell_{sch}^{p}(\G)}^2\leq \left(\lv H_{1}^{\ast}H_{1}\rv_{\ell_{sch}^{p}(\G)}+\lv H_{2}^{\ast}H_{2}\rv_{\ell_{sch}^{p}(\G)}\right)^2\\
&\overset{\eqref{Holder}}{\le} \left(\lv H_{1}^{\ast}\rv_{\ell_{sch}^{2p}(\G)}\lv H_{1}\rv_{\ell_{sch}^{2p}(\G)}+\lv H_{2}^{\ast}\rv_{\ell_{sch}^{2p}(\G)}\lv H_{2}\rv_{\ell_{sch}^{2p}(\G)}\right)^2\overset{\rm{Prop.} \ref{adjoint}}{=} \left(\lv H_{1}\rv_{\ell_{sch}^{2p}(\G)}^2+\lv H_{2}\rv_{\ell_{sch}^{2p}(\G)}^2\right)^2,
\end{split}\end{equation}
and
\begin{equation}\begin{split}\label{id2}
&\lv B\rv_{\ell_{sch}^{p}(\G)}^2=\lv H_{1}^{\ast}H_{2}+H_{2}^{\ast}H_{1}\rv_{\ell_{sch}^{p}(\G)}^2\leq \left(\lv H_{1}^{\ast}H_{2}\rv_{\ell_{sch}^{p}(\G)}+\lv H_{2}^{\ast}H_{1}\rv_{\ell_{sch}^{p}(\G)}\right)^2\\
&\overset{\eqref{Holder}}{\le} \left(\lv H_{1}^{\ast}\rv_{\ell_{sch}^{2p}(\G)}\lv H_{2}\rv_{\ell_{sch}^{2p}(\G)}+\lv H_{2}^{\ast}\rv_{\ell_{sch}^{2p}(\G)}\lv H_{1}\rv_{\ell_{sch}^{2p}(\G)}\right)^2\overset{\rm{Prop.} \ref{adjoint}}{=} 4\lv H_{1}\rv_{\ell_{sch}^{2p}(\G)}^2\lv H_{2}\rv_{\ell_{sch}^{2p}(\G)}^2.
\end{split}\end{equation}

Hence, combining \eqref{id1} and \eqref{id2} together with \eqref{equality}, we have 

\begin{equation*}\begin{split}
&\frac{1}{2}\Bigl(\lv H_{1}+H_{2}\rv_{\ell_{sch}^{2p}(\widehat{G})}^{2p}+\lv H_{1}-H_{2}\rv_{\ell_{sch}^{2p}(\widehat{G})}^{2p}\Bigr)\\
&\leq \left(\left(\lv H_{1}\rv_{\ell_{sch}^{2p}(\G)}^2+\lv H_{2}\rv_{\ell_{sch}^{2p}(\G)}^2\right)^2+4C_{p}\lv H_{1}\rv_{\ell_{sch}^{2p}(\G)}^2\lv H_{2}\rv_{\ell_{sch}^{2p}(\G)}^2\right)^{\frac{p}{2}}\\
&=\left(\lv H_{1}\rv_{\ell_{sch}^{2p}(\G)}^4+\lv H_{2}\rv_{\ell_{sch}^{2p}(\G)}^4+(4C_{p}+2)\lv H_{1}\rv_{\ell_{sch}^{2p}(\G)}^2\lv H_{2}\rv_{\ell_{sch}^{2p}(\G)}^2\right)^{\frac{p}{2}}\\
&\leq \left(\lv H_{1}\rv_{\ell_{sch}^{2p}(\G)}^4+(1+2C_{p})^2\lv H_{2}\rv_{\ell_{sch}^{2p}(\G)}^4+(4C_{p}+2)\lv H_{1}\rv_{\ell_{sch}^{2p}(\G)}^2\lv H_{2}\rv_{\ell_{sch}^{2p}(\G)}^2\right)^{\frac{p}{2}}\\
&\leq \left(\left(\lv H_{1}\rv_{\ell_{sch}^{2p}(\G)}^2+(1+2C_{p})\lv H_{2}\rv_{\ell_{sch}^{2p}(\G)}^2\right)^2\right)^{\frac{p}{2}}=\left(\lv H_{1}\rv_{\ell_{sch}^{2p}(\G)}^2+(1+2C_{p})\lv H_{2}\rv_{\ell_{sch}^{2p}(\G)}^2\right)^{p}.
\end{split}\end{equation*}
Therefore, we proved an analogue of \ref{ineq1} for $2p$ in the form
\begin{equation*}\begin{split}
\left(\frac{1}{2}\Bigl(\lv H_{1}+H_{2}\rv_{\ell_{sch}^{2p}(\widehat{G})}^{2p}+\lv H_{1}-H_{2}\rv_{\ell_{sch}^{2p}(\widehat{G})}^{2p}\Bigr)\right)^{1/2p}\leq\left(\lv H_{1}\rv_{\ell_{sch}^{2p}(\G)}^2+(1+2C_{p})\lv H_{2}\rv_{\ell_{sch}^{2p}(\G)}^2\right)^{\frac{1}{2}},
\end{split}\end{equation*} 
and obviously $C_{2p}\leq 2C_{p}+1.$ Note that $C_{2}=1$ by the parallelogram identity. Hence, it is easy to prove by induction that $C_{2^n}\leq 2 C_{2^{n-1}}+1\leq 2\left(2^{n-1}-1\right)+1=2^{n}-1$.

Now assume that $2\leq p<\infty$, then there exists $n\in\mathbb{N}$ such that $2^{n}\leq p<2^{n+1}.$ Choose $\theta\in(0,1)$ such that $1/p=(1-\theta)/2^{n}+\theta/2^{n+1}$. Since \ref{ineq1} holds for $2^n$ , it follows that an operator $T_{1}$ from $\ell_{sch}^{2^n}(\G)\oplus_{2,2^{n}-1}\ell_{sch}^{2^n}(\G)$ into $\ell_{sch}^{2^n}(\G)\oplus_{2^{n}}\ell_{sch}^{2^n}(\G)$ defined as 
$$T_{1}(H_{1},H_{2})(\xi,\xi):=\left(H_{1}(\xi)+H_{2}(\xi), H_{1}(\xi)-H_{2}(\xi)\right)$$
is bounded with norm 
\begin{equation}\label{norm 1 int}
\lv T_{1}\rv_{\ell_{sch}^{2^n}(\G)\oplus_{2,2^{n}-1}\ell_{sch}^{2^n}(\G)\rightarrow\ell_{sch}^{2^n}(\G)\oplus_{2^{n}}\ell_{sch}^{2^n}(\G)}\leq 2^{1/2^{n}}.    
\end{equation}
Similarly, since \ref{ineq1} holds for $2^{n+1},$ it follows that an operator $T_{2}$ from $\ell_{sch}^{2^{n+1}}(\G)\oplus_{2,2^{n+1}-1}\ell_{sch}^{2^{n+1}}(\G)$ into $\ell_{sch}^{2^{n+1}}(\G)\oplus_{2^{n+1}}\ell_{sch}^{2^{n+1}}(\G)$
defined by the same action as $T_{1}$ is bounded with norm 
\begin{equation}\label{norm 2 int}
\lv T_{2}\rv_{\ell_{sch}^{2^{n+1}}(\G)\oplus_{2,2^{n+1}-1}\ell_{sch}^{2^{n+1}}(\G)\rightarrow\ell_{sch}^{2^{n+1}}(\G)\oplus_{2^{n+1}}\ell_{sch}^{2^{n+1}}(\G)}\leq 2^{\frac{1}{2^{n+1}}}. 
\end{equation}

On the other hand, by Proposition \ref{calderon interpolation 2}, one has 
\begin{equation}\label{norm 3 int}
\left(\ell_{sch}^{2^n}(\G)\oplus_{2,2^{n}-1}\ell_{sch}^{2^n}(\G), \ell_{sch}^{2^{n+1}}(\G)\oplus_{2,2^{n+1}-1}\ell_{sch}^{2^{n+1}}(\G)\right)_{\theta}=\ell_{sch}^{p}(\G)\oplus_{2, C_{p}}\ell_{sch}^{p}(\G),
\end{equation}
where $C_{p}=(2^{n}-1)^{1-\theta}(2^{n+1}-1)^{\theta}.$ Moreover, by Proposition \ref{calderon interpolation 1}, one has 
\begin{equation}\label{norm 4 int}
\left(\ell_{sch}^{2^n}(\G)\oplus_{2^{n}}\ell_{sch}^{2^n}(\G), \ell_{sch}^{2^{n+1}}(\G)\oplus_{2^{n+1}}\ell_{sch}^{2^{n+1}}(\G)\right)_{\theta}=\ell_{sch}^{p}(\G)\oplus_{p}\ell_{sch}^{p}(\G).
\end{equation}
Therefore, by \eqref{norm 3 int}, \eqref{norm 4 int} and the general interpolation result \cite[Corollary 5.5.4]{lofstrom1976interpolation} together with \eqref{norm 1 int}, \eqref{norm 2 int}, we have that an operator $T$ from $\ell_{sch}^{p}(\G)\oplus_{2, C_{p}}\ell_{sch}^{p}(\G)$ into $\ell_{sch}^{p}(\G)\oplus_{p}\ell_{sch}^{p}(\G)$ defined with same action as operator $T_{1}$ and $T_{2}$ is bounded for any $2\leq p<\infty$ with norm
$$\lv T\rv_{\ell_{sch}^{p}(\G)\oplus_{2, C_{p}}\ell_{sch}^{p}(\G)\rightarrow\ell_{sch}^{p}(\G)\oplus_{p}\ell_{sch}^{p}(\G)}\leq \left(2^{1/2^{n}}\right)^{1-\theta} \left(2^{1/2^{n+1}}\right)^{\theta}=2^{1/p},$$
which proves \ref{ineq1}. Moreover, 
$$C_{p}=(2^{n}-1)^{1-\theta}(2^{n+1}-1)^{\theta}\leq (2^{n+1}-1)^{1-\theta}(2^{n+1}-1)^{\theta}=2^{n+1}-1\leq 2p-1.$$

\ref{ineq2}. First note, by replacing $H_{1}$ and $H_{2}$ with $H_{1}+H_{2}$ and $H_{1}-H_{2}$, respectively, \ref{ineq2} is equivalent to the following 
\begin{equation}\begin{split}\label{new enough}
    \Bigl(\lv H_{1}+H_{2}\psrv^2+c_{p}\lv H_{1}-H_{2}\psrv^2\Bigr)^{1/2}\leq 2^{1/q}\left(\lv H_{1}\psrv^{p}+\lv H_{2}\psrv^{p}\right)^{1/p},
    \end{split}\end{equation}
In order to prove \ref{ineq2}, we prove \eqref{new enough}.

Define the map $S$ from $\ell_{sch}^{p}(\G)\oplus_{p}\ell_{sch}^{p}(\G)$ into $\ell_{sch}^{p}(\G)\oplus_{2, c_{p}}\ell_{sch}^{p}(\G)$ by setting
$$S(H_{1},H_{2})(\xi,\xi):=\left(H_{1}(\xi)+H_{2}(\xi),H_{1}(\xi)-H_{2}(\xi)\right).$$ 
Note that it is enough to show that 
\begin{equation}\label{enough to prove 2}
\lv S\rv_{\ell_{sch}^{p}(\G)\oplus_{p}\ell_{sch}^{p}(\G)\rightarrow \ell_{sch}^{p}(\G)\oplus_{2, c_{p}}\ell_{sch}^{p}(\G)}\leq 2^{1/q}.    
\end{equation}

By the similar argument as in the proof of Proposition \ref{new Clarkson type inequality}\ref{new first inequality}, it is easy to show that the adjoint operator $S^{\ast}$ from $\left(\ell_{sch}^{p}(\G)\oplus_{2, c_{p}}\ell_{sch}^{p}(\G)\right)^{\ast}$ into $\left(\ell_{sch}^{p}(\G)\oplus_{p}\ell_{sch}^{p}(\G)\right)^{\ast}$ is defined as 
$$S(H_{1}, H_{2})(\xi,\xi)=\left(H_{1}(\xi)+H_{2}(\xi), H_{1}(\xi)-H_{2}(\xi)\right).$$

Furthermore, by Proposition \ref{duality of the direct sum}, one has 
\begin{equation*}\begin{split}
\left(\ell_{sch}^{p}(\G)\oplus_{2, c_{p}}\ell_{sch}^{p}(\G)\right)^{\ast}&=\ell_{sch}^{q}(\G)\oplus_{2, \frac{1}{c_{p}}}\ell_{sch}^{q}(\G),\\
\left(\ell_{sch}^{p}(\G)\oplus_{p}\ell_{sch}^{p}(\G)\right)^{\ast}&=\ell_{sch}^{q}(\G)\oplus_{q}\ell_{sch}^{q}(\G).
\end{split}\end{equation*}
By \ref{ineq1}, we have 
\begin{equation*}\begin{split}
\lv S\rv&_{\ell_{sch}^{p}(\G)\oplus_{p}\ell_{sch}^{p}(\G)\rightarrow \ell_{sch}^{p}(\G)\oplus_{2, c_{p}}\ell_{sch}^{p}(\G)}=\lv S^{\ast}\rv_{\left(\ell_{sch}^{p}(\G)\oplus_{2, c_{p}}\ell_{sch}^{p}(\G)\right)^{\ast}\rightarrow\left(\ell_{sch}^{p}(\G)\oplus_{p}\ell_{sch}^{p}(\G)\right)^{\ast}}\\
&=\lv S^{\ast}\rv_{\ell_{sch}^{q}(\G)\oplus_{2, \frac{1}{c_{p}}}\ell_{sch}^{q}(\G)\rightarrow\ell_{sch}^{q}(\G)\oplus_{q}\ell_{sch}^{q}(\G)}= \lv S^{\ast}\rv_{\ell_{sch}^{q}(\G)\oplus_{2, C_{q}}\ell_{sch}^{q}(\G)\rightarrow\ell_{sch}^{q}(\G)\oplus_{q}\ell_{sch}^{q}(\G)}\leq 2^{1/q},
\end{split}\end{equation*}
which proves \eqref{enough to prove 2}. 
Moreover, 
$$c_{p}=\frac{1}{C_{q}}\geq \frac{1}{2q-1}\geq \frac{p-1}{p+1},$$
completing the proof.
\end{proof}

Theorem \ref{mainth2} provides the following optimal rate estimates for $\delta_{\ell_{sch}^{p}(\widehat{G})}(\varepsilon)$ when $1<p<2$ and for $\rho_{\ell_{sch}^{p}(\widehat{G})}(t)$ when $2<p<\infty$ in a sense that together with \eqref{optimal rate in general}, one has $\delta_{\ell_{sch}^{p}(\G)}(\varepsilon)=O(\varepsilon^2),$ $0<\varepsilon<2,$ when $1<p<2$ and $\rho_{\ell_{sch}^{p}(\G)}(t)=O(t^2),$ $t>0,$ when $2<p<\infty$, similar to the ones in the classical case of Lebesgue spaces.

\begin{corollary}\label{optimal corollary}
    \begin{enumerate}[label=(\roman*)]
        \item If $1<p\leq 2,$ then for any $0<\varepsilon<2$, one has
        $$\delta_{\ell_{sch}^{p}(\widehat{G})}(\varepsilon)\geq \frac{c_{p}}{8}\varepsilon^2.$$
        \item If $2\leq p<\infty,$ then for any $t>0$, one has
        $$\rho_{\ell_{sch}^{p}(\widehat{G})}(t)\leq \frac{C_{p}}{2}t^2.$$
    \end{enumerate}
\end{corollary}

\begin{proof}
(i). Let $1<p\leq 2$. Then, by \ref{ineq2} of Theorem \ref{mainth2}, one has 
$$\Bigl(\lv K_{1}\prv^2+c_{p}\lv K_{2}\psrv^2\Bigr)^{1/2}
    \leq \left(\frac{1}{2}\left(\lv K_{1}+K_{2}\psrv^{p}+\lv K_{1}-K_{2}\psrv^{p}\right)\right)^{1/p}$$
for $K_{1},K_{2}\in \ell_{sch}^{p}(\widehat{G}).$ Denote $H_{1}=K_{1}+K_{2}\in \ell_{sch}^{p}(\widehat{G})$ and $H_{2}=K_{1}-K_{2}\in \ell_{sch}^{p}(\widehat{G}).$ Hence, the last inequality can be rewritten as 
\begin{equation*}\begin{split}
\Bigl(\lv \frac{H_{1}+H_{2}}{2}\psrv^2+c_{p}\lv \frac{H_{1}-H_{2}}{2}\psrv^2\Bigr)^{1/2}
\leq \left(\frac{1}{2}\left(\lv H_{1}\psrv^{p}+\lv H_{2}\psrv^{p}\right)\right)^{1/p}.
\end{split}\end{equation*}
Assuming that $\lv H_{1}\psrv=\lv H_{2}\psrv=1$ and $\lv H_{1}-H_{2}\psrv=\varepsilon,$ it follows that
$$\left(\lv \frac{H_{1}+H_{2}}{2}\psrv^2+\frac{c_{p}}{4}\varepsilon^2\right)^{1/2}\leq 1.$$
Therefore,
$$\frac{c_{p}}{4}\varepsilon^2\leq 1-\lv \frac{H_{1}+H_{2}}{2}\psrv^2=\left(1-\lv \frac{H_{1}+H_{2}}{2}\psrv\right)\left(1+\lv \frac{H_{1}+H_{2}}{2}\psrv\right)$$
$$\leq \left(1-\lv \frac{H_{1}+H_{2}}{2}\psrv\right)\left(1+\frac{\lv H_{1}\psrv+\lv H_{2}\psrv}{2}\right)=2\left(1-\lv \frac{H_{1}+H_{2}}{2}\psrv\right).$$
Hence, taking infimum over all $H_{1}, H_{2}\in \ell_{sch}^{p}(\widehat{G})$ with norm one and $\lv H_{1}-H_{2}\psrv=\varepsilon$, we finally have
$$\delta_{\ell_{sch}^{p}(\widehat{G})}\geq \frac{c_{p}}{8}\varepsilon^2$$

(ii). Let $2\leq p<\infty$ and $H_{1}, H_{2}\in \ell_{sch}^{p}(\widehat{G})$. Define a function $f(x)=x^{1/p},$ $x\geq 0.$ Assume that $x_{1}=\lv H_{1}+tH_{2}\psrv^{p}$ and $x_{2}=\lv H_{1}-tH_{2}\psrv^{p}$. Hence, since $f$ is a concave function, it follows that 

\begin{equation}\begin{split}\label{neweq1}
&\frac{\lv H_{1}+tH_{2}\psrv+\lv H_{1}-tH_{2}\psrv}{2}=\frac{f(x_{1})+f(x_{2})}{2}\\
&\leq f\left(\frac{x_{1}+x_{2}}{2}\right)=\left(\frac{\lv H_{1}+tH_{2}\psrv^{p}+\lv H_{1}-tH_{2}\psrv^p}{2}\right)^{1/p}. 
\end{split}\end{equation}
Moreover, by Theorem \ref{mainth2} \ref{ineq1}, one has
\begin{equation}\label{neweq2}
\left(\frac{\lv H_{1}+tH_{2}\psrv^{p}+\lv H_{1}-tH_{2}\psrv^p}{2}\right)^{1/p}\leq \left(\lv H_{1}\psrv^2+C_{p}\lv t H_{2}\psrv^2\right)^{1/2}.
\end{equation}
Assuming $\lv H_{1}\psrv=\lv H_{2}\psrv=1$, and combining \eqref{neweq1} together with \eqref{neweq2}, we have 

\begin{equation*}\begin{split}
\frac{\lv H_{1}+tH_{2}\psrv+\lv H_{1}-tH_{2}\psrv}{2}-1\leq \left(1+C_{p} t^2\right)^{1/2}-1\leq 1+ \frac{C_{p}}{2}t^2-1=\frac{C_{p}}{2}t^2,   
\end{split}\end{equation*} where the last inequality follows from the Bernoulli's inequality. 

Therefore, taking supremum over all $H_{1}, H_{2}\in \ell_{sch}^{p}(\widehat{G})$ with norm one, we finally have that
$$\rho_{\ell_{sch}^{p}(\widehat{G})}(t)\leq \frac{C_{p}}{2}t^{2},$$
completing the proof.
\end{proof}

One of the applications of above given estimates in Lemma \ref{lemma0} and Corollary \ref{optimal corollary} is the following noncommutative extension of a classical theorem of Orlicz concerning unconditionally convergent series (see, for example, \cite[Theorem 1.e.11]{lindenstrauss2013classical}). 

Let $X$ be a Banach space and $\{x_{n}\}_{n\geq 1}$ be a sequence of elements from $X$. We recall that a series $\sum_{j=1}^{\infty}x_{j}$ is unconditionally convergent in $X$, if $\sum_{j=1}^{\infty}\theta_{j}x_{j}$ converges in $X$ for all choices of signs $\theta_{j}=\pm1,$ $j\geq 1.$

\begin{corollary}
If $\sum\limits_{j=1}^{\infty}H_{j}$ is an unconditionally convergent series in $\ell_{sch}^{p}(\G),$ $1<p<\infty$, then
$$\sum_{j=1}^{\infty}\lv H_{j}\rv_{\ell_{sch}^{p}(\G)}^{max\{p,2\}}<\infty.$$
\end{corollary}

\begin{proof}
Let us first recall the theorem of Kadec \cite[Theorem 1]{kadec1956unconditional}, which says that if $(X, \lv \cdot\rv_{X})$ is a uniformly convex Banach space with modulus of convexity $\delta_{X}(\cdot)$ and if $\sum\limits_{j=1}^{\infty}x_{j}$ is unconditionally convergent, then 
$$\sum_{j=1}^{\infty}\delta_{X}\left(\lv x_{j}\rv_{X}\right)<\infty.$$

By Theorem \ref{mainth0}, we have that $\ell_{sch}^{p}(\G)$ is uniformly convex for $1<p<\infty.$ Therefore, by Kadec's Theorem, one has 
\begin{equation*}
\sum_{j=1}^{\infty}\delta_{\ell_{sch}^{p}(\G)}\left(\lv H_{j}\rv_{\ell_{sch}^{p}(\G)}\right)<\infty
\end{equation*} for $\{H_{j}\}_{j\geq 1}\subset\ell_{sch}^{p}(\G).$
Hence, by Corollary \ref{optimal corollary}, one has 
\begin{equation}\label{kadcoroleq1}
\sum_{j=1}^{\infty}\frac{c_{p}}{8}\lv H_{j}\rv_{\ell_{sch}^{p}(\G)}^2\leq\sum_{j=1}^{\infty}\delta_{\ell_{sch}^{p}(\G)}\left(\lv H_{j}\rv_{\ell_{sch}^{p}(\G)}\right)<\infty,
\end{equation} for $1<p\leq 2$, where $c_{p}$ is the constant from Theorem \ref{mainth2} \ref{ineq2}.
Moreover, by Lemma \ref{lemma0} \ref{lemma0 inequality 2}, one has 
\begin{equation}\label{kadcoroleq2}
\sum_{j=1}^{\infty}\frac{1}{p\cdot 2^{p}}\lv H_{j}\rv_{\ell_{sch}^{p}(\G)}^{p}\leq\sum_{j=1}^{\infty}\delta_{\ell_{sch}^{p}(\G)}\left(\lv H_{j}\rv_{\ell_{sch}^{p}(\G)}\right)<\infty,
\end{equation} for $2<p< \infty.$
Therefore, combining \eqref{kadcoroleq1} and \eqref{kadcoroleq2}, we finally have 
$$\sum_{j=1}^{\infty}\lv H_{j}\psrv^{\max\{2,p\}}<\infty,$$
for $1<p<\infty.$
\end{proof}

Finally, by Proposition \ref{new Clarkson type inequality} and Theorem \ref{mainth2}, we can identify the type and cotype (cf. Definition \ref{type def}) of the Banach space $\ell_{sch}^{p}(\G)$ for $1<p<\infty$. Note that the result is analogous with the result for classical $\ell^{p}$-spaces.

\begin{theorem}\label{type}
Let $1<p<\infty.$ Then $\ell_{sch}^{p}(\G)$ is of type $\min\{2,p\}$ and cotype $\max\{2,p\}$. In particular, let $\{H_{j}\}_{j=1}^{n}\subset\ell_{sch}^{p}(\G)$ be a finite sequence. If $1<p\leq 2$, then
\begin{equation}\label{first type inequality}
\sqrt{c_{p}}\left(\sum_{j=1}^{n}\lv H_{j}\rv_{\ell_{sch}^{p}(\G)}^2\right)^{\frac{1}{2}}\leq \left(\int_{0}^{1}\lv \sum_{j=1}^{n}r_{j}(t)H_{j}\rv_{\ell_{sch}^{p}(\G)}^2 dt \right)^{1/2}\leq \left(\sum_{j=1}^{n}\lv H_{j}\rv_{\ell_{sch}^{p}(\G)}^{p}\right)^{\frac{1}{p}},
\end{equation} where $c_{p}$ is the constant given in Theorem \ref{mainth2} \ref{ineq2}.
If $2\leq p<\infty$, then
\begin{equation}\label{second type inequality}
\left(\sum_{j=1}^{n}\lv H_{j}\rv_{\ell_{sch}^{p}(\G)}^p\right)^{\frac{1}{p}}\leq \left(\int_{0}^{1}\lv \sum_{j=1}^{n}r_{j}(t)H_{j}\rv_{\ell_{sch}^{p}(\G)}^2 dt \right)^{1/2}\leq \sqrt{C_{p}}\left(\sum_{j=1}^{n}\lv H_{j}\rv_{\ell_{sch}^{p}(\G)}^{2}\right)^{\frac{1}{2}},
\end{equation} where $C_{p}$ is the constant given in Theorem \ref{mainth2} \ref{ineq1}.
\end{theorem}

\begin{proof}
We first prove the given inequalities, which together with Remark \ref{different averages of type} imply the type and cotype properties of the space $\ell_{sch}^{p}(\G)$, $1<p<\infty$ (cf. Definition \ref{type def}). 

Let $1<p\leq 2$ and $q$ be the conjugate index of $p.$ We first prove the second inequality of \eqref{first type inequality}. First note that since every $r_{j}(t),$ $1\leq j\leq n-1,$ has identical values on both intervals $[\frac{2k}{2^{n}},\frac{2k+1}{2^{n}}]$ and $[\frac{2k+1}{2^{n}},\frac{2k+2}{2^{n}}]$ for each integer $0\leq k\leq 2^{n-1}-1$, and while $r_{n}(t)$ equals $1$ and $-1$ on these intervals, respectively, it follows that 
$$\int_{\frac{2k}{2^{n}}}^{\frac{2k+1}{2^{n}}}\lv \sum_{j=1}^{n-1}r_{j}(t)H_{j}\pm H_{n}\psrv^{q}dt=\int_{\frac{2k+1}{2^{n}}}^{\frac{2k+2}{2^{n}}}\lv \sum_{j=1}^{n-1}r_{j}(t)H_{j}\pm H_{n}\psrv^{q}dt,\quad 0\leq k\leq 2^{n-1}-1.$$
Hence, by the last equality, we have 
\begin{equation}\begin{split}\label{typeq1}
\int_{0}^{1}&\lv \sum_{j=1}^{n}r_{j}(t)H_{j}\psrv^{q}dt\\
&=\sum_{k=0}^{2^{n-1}-1}\left(\int_{\frac{2k}{2^{n}}}^{\frac{2k+1}{2^{n}}}\lv \sum_{j=1}^{n-1}r_{j}(t)H_{j}+H_{n}\psrv^{q}dt+\int_{\frac{2k+1}{2^{n}}}^{\frac{2k+2}{2^{n}}}\lv \sum_{j=1}^{n-1}r_{j}(t)H_{j}-H_{n}\psrv^{q}dt\right)\\
&=\frac{1}{2}\sum_{k=0}^{2^{n-1}-1}\left(\int_{\frac{2k}{2^{n}}}^{\frac{2k+1}{2^{n}}}\lv \sum_{j=1}^{n-1}r_{j}(t)H_{j}+H_{n}\psrv^{q}dt+\int_{\frac{2k+1}{2^{n}}}^{\frac{2k+2}{2^{n}}}\lv \sum_{j=1}^{n-1}r_{j}(t)H_{j}-H_{n}\psrv^{q}dt\right)\\
&+\frac{1}{2}\sum_{k=0}^{2^{n-1}-1}\left(\int_{\frac{2k+1}{2^{n}}}^{\frac{2k+2}{2^{n}}}\lv \sum_{j=1}^{n-1}r_{j}(t)H_{j}+H_{n}\psrv^{q}dt+\int_{\frac{2k}{2^{n}}}^{\frac{2k+1}{2^{n}}}\lv \sum_{j=1}^{n-1}r_{j}(t)H_{j}-H_{n}\psrv^{q}dt\right)\\
&=\frac{1}{2}\int_{0}^{1}\left(\lv \sum_{j=1}^{n-1}r_{j}(t)H_{j}+H_{n}\psrv^{q}+\lv \sum_{j=1}^{n-1}r_{j}(t)H_{j}-H_{n}\psrv^{q}\right)dt.
\end{split}\end{equation}
Moreover, by Proposition \ref{new Clarkson type inequality} \ref{new first inequality}, one has 
\begin{equation}\begin{split}\label{typeq2}
\frac{1}{2}\int_{0}^{1}&\left(\lv \sum_{j=1}^{n-1}r_{j}(t)H_{j}+H_{n}\psrv^{q}+\lv \sum_{j=1}^{n-1}r_{j}(t)H_{j}-H_{n}\psrv^{q}\right)dt\\
&\leq \int_{0}^{1}\left(\lv \sum_{j=1}^{n-1}r_{j}(t)H_{j}\psrv^{p}+\lv H_{n}\psrv^{p}\right)^{\frac{q}{p}}dt\\
&\leq \left(\left(\int_{0}^{1}\lv \sum_{j=1}^{n-1}r_{j}(t)H_{j}\psrv^{q}dt\right)^{\frac{p}{q}}+\left(\int_{0}^{1}\lv H_{n}\psrv^{q}\right)^{\frac{p}{q}}dt\right)^{\frac{q}{p}}\\
&=\left(\left(\int_{0}^{1}\lv \sum_{j=1}^{n-1}r_{j}(t)H_{j}\psrv^{q}dt\right)^{\frac{p}{q}}+\lv H_{n}\psrv^{p}\right)^{\frac{q}{p}},
\end{split}\end{equation} where the last inequality follows from the Minkowski inequality in the classical $L^{q/p}[0,1]$-space with $q/p\geq 1$. Hence, combining \eqref{typeq1} and \eqref{typeq2}, and repeating the same argument $n-1$ more times, we have 
\begin{equation*}\begin{split}
\left(\int_{0}^{1}\lv \sum_{j=1}^{n}r_{j}(t)H_{j}\psrv^{q}dt \right)^\frac{p}{q}&\leq \left(\int_{0}^{1}\lv \sum_{j=1}^{n-1}r_{j}(t)H_{j}\psrv^{q}dt\right)^{\frac{p}{q}}+\lv H_{n}\psrv^{p}\\
&\leq \left(\int_{0}^{1}\lv \sum_{j=1}^{n-2}r_{j}(t)H_{j}\psrv^{q}dt\right)^{\frac{p}{q}}+\lv H_{n-1}\psrv^{p}+\lv H_{n}\psrv^{p}\\
&\ldots\\
&=\left(\int_{0}^{1}\lv r_{1}(t)H_{1}\psrv^{q}dt\right)^{\frac{p}{q}}+\sum_{j=2}^{n}\lv H_{j}\psrv^{p}\\
&=\left(\frac{1}{2}\int_{0}^{1}\left(\lv -H_{1}\psrv^{q}+\lv H_{1}\psrv^{q}\right)dt\right)^{\frac{p}{q}}+\sum_{j=2}^{n}\lv H_{j}\psrv^{p}\\
&=\sum_{j=1}^{n}\lv H_{j}\psrv^{p}.
\end{split}\end{equation*}
Therefore, 
\begin{equation}\label{typeq3}
\left(\int_{0}^{1}\lv \sum_{j=1}^{n}r_{j}(t)H_{j}\psrv^{q}dt \right)^\frac{1}{q}\leq \left(\sum_{j=1}^{n}\lv H_{j}\psrv^{p}\right)^{\frac{1}{p}}. 
\end{equation}
On the other hand, since $q\geq 2$, the classical embedding result of $L^{p}$-spaces, applied for the function $g(t)=\lv \sum\limits_{j=1}^{n}r_{j}(t)H_{j}\psrv$, implies that 
\begin{equation}\label{typeq4}
\left(\int_{0}^{1}\lv \sum_{j=1}^{n}r_{j}(t)H_{j}\psrv^{2} dt\right)^\frac{1}{2}\leq \left(\int_{0}^{1}\lv \sum_{j=1}^{n}r_{j}(t)H_{j}\psrv^{q}dt\right)^{\frac{1}{q}}.
\end{equation}
Finally, combining \eqref{typeq3} and \eqref{typeq4}, one has the second inequality of \eqref{first type inequality}.

We now consider the first inequality of \eqref{first type inequality}. By \eqref{typeq1} and Theorem \ref{mainth2} \ref{ineq2}, one has 

\begin{equation*}\begin{split}
\int_{0}^{1}\lv \sum_{j=1}^{n}r_{j}(t)H_{j}\psrv^{p}dt&=\frac{1}{2}\int_{0}^{1}\left(\lv \sum_{j=1}^{n-1}r_{j}H_{j}+H_{n}\psrv^{p}+\lv \sum_{j=1}^{n-1}r_{j}(t)H_{j}-H_{n}\psrv^{p}\right)dt \\
&\geq \int_{0}^{1}\left(\lv \sum_{j=1}^{n-1}r_{j}(t)H_{j}\psrv^{2}+c_{p}\lv H_{n}\psrv^{2}\right)^{p/2}dt\\
&\geq \left(\left(\int_{0}^{1}\lv \sum_{j=1}^{n-1}r_{j}(t)H_{j}\psrv^{p}dt\right)^{\frac{2}{p}}+\left(\int_{0}^{1}c_{p}^{p/2}\lv H_{n}\psrv^{p}dt\right)^{\frac{2}{p}}\right)^{\frac{p}{2}}\\
&=\left(\left(\int_{0}^{1}\lv \sum_{j=1}^{n-1}r_{j}(t)H_{j}\psrv^{p}dt\right)^{\frac{2}{p}}+c_{p}\lv H_{n}\psrv^{2}\right)^{\frac{p}{2}},
\end{split}\end{equation*} where the last inequality follows from the reverse Minkowski inequality in the classical $L^{p/2}[0,1]$-space with $p/2\leq 1$. Hence, repeating the same argument $n-1$ more times, we have
\begin{multline*}
\left(\int_{0}^{1}\lv \sum_{j=1}^{n}r_{j}(t)H_{j}\psrv^{p}dt\right)^{\frac{2}{p}}
\geq \left(\int_{0}^{1}\lv \sum_{j=1}^{n-1}r_{j}(t)H_{j}\psrv^{p}dt\right)^{\frac{2}{p}}+c_{p}\lv H_{n}\psrv^{2}\\
\geq \left(\int_{0}^{1}\lv \sum_{j=1}^{n-2}r_{j}(t)H_{j}\psrv^{p}dt\right)^{\frac{2}{p}}+c_{p}\lv H_{n-1}\psrv^{2}+c_{p}\lv H_{n}\psrv^{2}\\
\ldots\\
\geq \left(\int_{0}^{1}\lv r_{1}(t)H_{1}\psrv^{p}dt\right)^{\frac{2}{p}}+c_{p}\sum_{j=2}^{n}\lv H_{j}\psrv^{2}\\
=\left(\frac{1}{2}\int_{0}^{1}\left(\lv -H_{1}\psrv^{p}+\lv H_{1}\psrv^{p}\right)dt\right)^{\frac{2}{p}}+c_{p}\sum_{j=2}^{n}\lv H_{j}\psrv^{2}=c_{p}\sum_{j=1}^{n}\lv H_{j}\psrv^{2}.
\end{multline*}
Therefore,
\begin{equation}\label{typeq5}
    \left(\int_{0}^{1}\lv \sum_{j=1}^{n}r_{j}(t)H_{j}\psrv^{p}dt\right)^{\frac{1}{p}}\geq \sqrt{c_{p}}\left(\sum_{j=1}^{n}\lv H_{j}\psrv^{2}\right)^{\frac{1}{2}}. 
\end{equation}
On the other hand, since $p\leq 2$, the classical embedding result of $L^{p}$-spaces, applied for the function $g(t)=\lv \sum\limits_{j=1}^{n}r_{j}(t)H_{j}\psrv$, implies that 
\begin{equation}\label{typeq6}
\left(\int_{0}^{1}\lv \sum_{j=1}^{n}r_{j}(t)H_{j}\psrv^{2} dt\right)^\frac{1}{2}\geq \left(\int_{0}^{1}\lv \sum_{j=1}^{n}r_{j}(t)H_{j}\psrv^{p}dt\right)^{\frac{1}{p}}.
\end{equation}
Finally, combining \eqref{typeq5} and \eqref{typeq6}, one has the first inequality of \eqref{first type inequality}.

Two inequalities in \eqref{second type inequality} can be proved using the similar arguments as above together with Theorem \ref{new Clarkson type inequality} \ref{new second inequality} and Theorem \ref{mainth2} \ref{ineq2}, hence, we omit it.

Finally, the inequalities \eqref{first type inequality} and \eqref{second type inequality} together with Remark \ref{different averages of type} imply the type and cotype assertions.
\end{proof}

\section{Family of $\ell^{p}$-spaces $\ell^{p}(\widehat{G})$ based on Hilbert-Schmidt ideal}

In this section, we prove similar geometric properties as in the previous section for the space $\ell^{p}(\widehat{G}),$ $1< p<\infty.$ One of the main results of this section is the following result similar to Theorem \ref{mainth0}. 

\begin{theorem}\label{mainth1} The space $\ell^{p}(\widehat{G})$ is uniformly convex and uniformly smooth for $1<p<\infty$.
\end{theorem}

Note that the reflexivity of $\ell^{p}(\G)$ for $1<p<\infty$ also follows from Proposition \ref{MilmanPettis}. Another proof of this fact is given in \cite[Theorem 2]{tulenov2018clarkson}. To prove Theorem \ref{mainth1}, we first present the analogue of Lemma \ref{lemma0}. 

\begin{lemma}\label{lemma1}
Let $1<p<\infty$, $0<\varepsilon\leq2$ and $t>0.$ Let $q$ be a conjugate index of $p$, i.e., $\frac{1}{p}+\frac{1}{q}=1.$
\begin{enumerate}[label=(\roman*)]
    \item \label{lemma1 inequality 1} If $1<p\leq 2,$ then 
    $$\delta_{\ell^{p}(\widehat{G})}(\varepsilon)\geq \frac{\varepsilon^{q}}{q\cdot 2^{q}},\quad \rho_{\ell^{p}(\widehat{G})}(t)\leq \frac{t^{p}}{p}.$$
    \item \label{lemma1 inequality 2}If $2<p< \infty,$ then 
    $$\delta_{\ell^{p}(\widehat{G})}(\varepsilon)\geq \frac{\varepsilon^{p}}{p\cdot 2^{p}},\quad \rho_{\ell^{p}(\widehat{G})}(t)\leq \frac{t^{q}}{q}.$$
\end{enumerate}
\end{lemma}

\begin{proof}
The argument mainly follows the same line as of reasoning of the proof of Lemma \ref{lemma0}, hence, we omit it.
\end{proof}

\begin{proof}[Proof of Theorem \ref{mainth1}] Let $1<p\leq 2$. Then, by Lemma \ref{lemma1} \ref{lemma1 inequality 1}, one has 
$$\delta_{\ell^{p}(\widehat{G})}(\varepsilon)\leq \frac{\varepsilon^{q}}{q\cdot 2^{q}},\quad \forall \varepsilon>0$$
and
$$\lim_{t\rightarrow0}\frac{\rho_{\ell^{p}(\widehat{G})}}{t}\leq \lim_{t\rightarrow0}\frac{t^{p-1}}{p}=0.$$
Hence, $\ell^{p}(\widehat{G})$ is uniformly convex and uniformly smooth for $1<p
\leq 2$. The case when $2<p<\infty$ analogously follows from Lemma \ref{lemma1} \ref{lemma1 inequality 2}, hence, we omit it.
\end{proof}

\begin{remark}
Note that using the exact same method as in the proof of Theorem \ref{mainth2}, one can also prove similar inequalities with the same constants $C_{p}$ and $c_{p}$ for $\ell^{p}(\G),$ $1<p<\infty,$ spaces based on the Hilbert-Schmidt ideal. Therefore, combining this result and Clarkson inequalities from \cite[Theorem 3]{tulenov2018clarkson} (or see Proposition \ref{Clarkson type inequality}), one can also apply the methods of Theorem \ref{type} and obtain that the space $\ell^{p}(\G)$, $1<p<\infty,$ is of type $\min\{2, p\}$ and cotype $\max\{2, p\}.$      
\end{remark}

\subsection*{Acknowledgements}

The first author was partially supported by the grant No. AP22784356 of the Science Committee of the Ministry of Science and Higher Education of the Republic of Kazakhstan.
The second author was supported by the EPSRC grants EP/R003025/2 and EP/V005529/1 and by FWO grant G011522N. The second author was also supported by Odysseus and Methusalem grants (01M01021 (BOF Methusalem) and 3G0H9418 (FWO Odysseus)) at Ghent University.

\end{document}